\documentclass[11pt]{amsart}
\usepackage{latexsym, mathrsfs, color, multirow,bbm,mathtools,amsmath}
\usepackage{graphics}
\usepackage[flushleft]{threeparttable}
\usepackage{subcaption}
\usepackage{xparse}
\usepackage{enumitem}\usepackage{adjustbox}
\input{xy}
\xyoption{all}

\usepackage[colorlinks=true, pdfstartview=FitV,
 linkcolor=black,citecolor=black,urlcolor=black]{hyperref}
 
\usepackage{tikz-cd}
\usepackage{tkz-euclide}
\usepackage{algorithm}
\usepackage{algpseudocode}

\usepackage[
backend = biber,
style = nature,]{biblatex}
\numberwithin{equation}{section}
\theoremstyle{plain}
\newtheorem{Thm}[equation]{Theorem}
\newtheorem{Prop}[equation]{Proposition}
\newtheorem{Cor}[equation]{Corollary}
\newtheorem{Lem}[equation]{Lemma}

\theoremstyle{definition}
\newtheorem{Def}[equation]{Definition}

\newtheorem{Rmk}[equation]{Remark}

\newenvironment{red}{\relax\color{red}}{\relax}
\newenvironment{blue}{\relax\color{blue}}{\hspace*{.5ex}\relax}

\newcommand{\ber}{\begin{red}}
\newcommand{\er}{\end{red}}
\newcommand{\beb}{\begin{blue}}
\newcommand{\eb}{\end{blue}}

\newcommand{\bw}{{\boldsymbol{w}}}

\begin{document}
\title{New Hereditary and Mutation-Invariant Properties Arising from Forks}
\author[T. J. Ervin]{Tucker J. Ervin}
\address{Department of Mathematics, University of Alabama,
	Tuscaloosa, AL 35487, U.S.A.}
\email{tjervin@crimson.ua.edu}


\begin{abstract}
A hereditary property of quivers is a property preserved by restriction to any full subquiver.
Similarly, a mutation-invariant property of quivers is a property preserved by mutation.
Using forks, a class of quivers developed by M. Warkentin, we introduce a new hereditary and mutation-invariant property.
We prove that a quiver being mutation-equivalent to a finite number of non-forks --- defined as having a finite forkless part --- is this new property, using only elementary methods.
Additionally, we show that a more general property --- having a finite pre-forkless part --- is also a new hereditary and mutation-invariant property in much the same manner.
\end{abstract}

\maketitle

\section{Introduction} \label{sec-introduction}

Properties of quivers and quiver mutations are important objects of study in terms of cluster algebras and in their own right. 
Those that are hereditary and mutation-invariant are some of the most interesting, as they are preserved under both mutation and taking any full subquiver.
The most comprehensive list of hereditary and mutation-invariant properties we are aware of  --- to which we have added a few properties --- comes from Fomin, Igusa, and Lee's paper on universal quivers \cite[Example 5.2]{fomin_universal_2021}.
Let $Q$ be a quiver.
Then the following are hereditary and mutation-invariant properties:
\begin{itemize}
    \item $Q$ has finite type \cite{fomin_cluster_2003}, \cite{fomin_introduction_2021};

    \item $Q$ has finite mutation type \cite{fomin_introduction_2021};

    \item $Q$ is a full subquiver of a certain infinite quiver \cite{henrich_mutation_2011};

    \item $Q$ has finite/tame representation type \cite{holm_cluster_2010};

    \item Any quiver mutation-equivalent to $Q$ has arrow multiplicities $\leq k$ (for some $k \in \mathbb{Z}_{>0}$)\footnote{Note that this implies $Q$ is mutation-finite} \cite{fomin_universal_2021};

    \item $Q$ is mutation-acyclic \cite{fomin_universal_2021};

    \item The mutation class of $Q$ is embeddable into a mutation class \cite{fomin_introduction_2021};

    \item $Q$ admits a reddening sequence \cite{muller_existence_2016};

    \item The arrow multiplicities in $Q$ are all divisible by $d$ (for some $d \in \mathbb{Z}_{> 0}$)\footnote{It is a hereditary property by definition, and it is easily shown to be mutation-invariant};

    \item Any quiver mutation-equivalent to $Q$ has $k$ or more arrows between each pair of vertices (for some $k \in \mathbb{Z}_{>0}$)\footnote{It is a mutation-invariant property by definition, and it is easily shown to be hereditary. Thanks to Scott Neville for the name}. If $k = 2$, then we say that $Q$ is \textit{mutation-abundant}.
\end{itemize}
We will add to this list with Theorems \ref{thm-forkless-quiver-intro} and \ref{thm-pre-forkless-quiver-intro}.

In 2014, Matthias Warkentin's dissertation presented a special type of quiver, called a \textit{fork} \cite{warkentin_exchange_2014}. 
One of the great advantages of forks is that their properties can be proven using only combinatorial methods.
This makes them of particular interest when studying quivers and quiver mutation without the use of category theory or representation theory.
Warkentin was able to show several important results in this way: 
\begin{itemize}
    \item Every fork is the root of a tree in its mutation graph \cite[Lemma 2.8]{warkentin_exchange_2014};
    
    \item A connected quiver is mutation-infinite if and only if it is mutation-equivalent to a fork \cite[Theorem 3.2]{warkentin_exchange_2014};

    \item The ratio of non-forks to forks in the mutation graph of a connected mutation-infinite quiver tends to 0 as we mutate farther away from our original quiver \cite[Proposition 5.4]{warkentin_exchange_2014}.
\end{itemize}

The last property inspired this paper, producing the following two questions.
Are there quivers that are mutation-equivalent to a finite number of non-forks?
Equivalently, are there quivers that have a finite ``forkless'' part?
If such a quiver exists, this property is naturally a mutation-invariant one.
Is it also a hereditary property?
The first question was answered affirmatively by Warkentin for an example on 3 vertices and a more general family of quivers, but we will answer the second question with the following theorem.

\begin{Thm} \label{thm-forkless-quiver-intro}
    Let $Q$ be a quiver that has a finite forkless part.
    Then any (possibly disconnected) full subquiver $Q'$ also has a finite forkless part.
    Furthermore, if $M$ is the number of non-forks mutation-equivalent to $Q$ and $m$ is the number of vertices in $Q'$, then the number of non-forks mutation-equivalent to $Q'$ is bounded by $\max\{M, m\}$.
\end{Thm}

This theorem also generalizes nicely to Warkentin's \textit{pre-forks}, which are themselves generalizations of forks.

\begin{Thm} \label{thm-pre-forkless-quiver-intro}
    Let $Q$ be a quiver that has a finite pre-forkless part.
    Then any (possibly disconnected) full subquiver $Q'$ also has a finite pre-forkless part.
    Furthermore, if $M$ is the number of non-pre-forks mutation-equivalent to $Q$ and $m$ is the number of vertices in $Q'$, then the number of non-pre-forks mutation-equivalent to $Q'$ is bounded by $\max\{M,10m-20\}$.
\end{Thm}

In Section \ref{sec-preliminaries}, we go over the preliminary material needed, including quivers, forks, and their properties.
Then we go on to prove Theorem~\ref{thm-forkless-quiver-intro} in Section \ref{sec-forkless-part}.
Section \ref{sec-keys-pre-forks} introduces us to pre-forks and quivers we will call keys.
We then dive into Section \ref{sec-fpfp}, which proves Theorem~\ref{thm-pre-forkless-quiver-intro} in an analogous way to the proof of Theorem~\ref{thm-forkless-quiver-intro}.
Additionally, we finish by demonstrating the Venn diagram of hereditary and mutation-invariant properties given by Figure~\ref{fig-venn-diagram}.

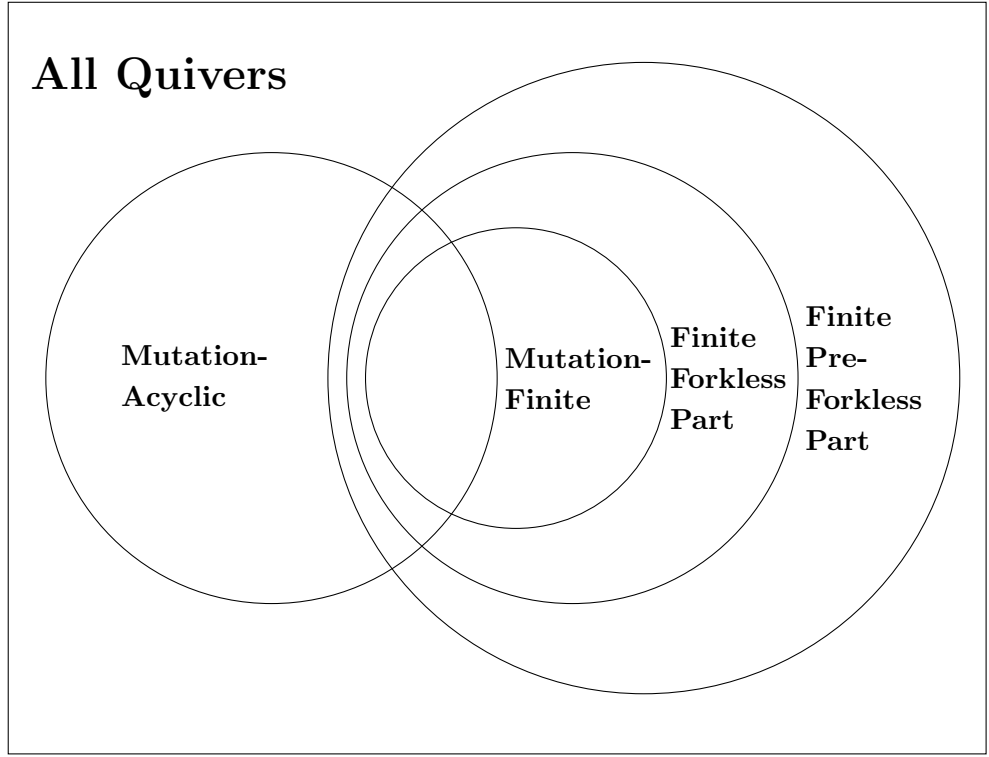
\begin{figure}
    \centering
    \begin{tikzpicture}
    	\begin{scope}
        \draw (-5,5) rectangle (8,-5);
        \draw[draw = black] (-1.5,0) circle (3);
        \draw[draw = black] (1.75,0) circle (2);
        \draw[draw = black] (2.5,0) circle (3);
        \draw[draw = black] (3.45,0) circle (4.2);
        
        \node at (-3,4) {\LARGE\textbf{All Quivers}};
        \node[text width=2cm] at (-2.5,0) {\textbf{Mutation-Acyclic}};
        \node[text width=2cm] at (2.6,0) {\textbf{Mutation-Finite}};
        \node[text width=2cm] at (4.8,0) {\textbf{Finite Forkless Part}};
        \node[text width=2cm] at (6.6,0) {\textbf{Finite Pre-Forkless Part}};
        
        \end{scope}
    \end{tikzpicture}
    \caption{Venn Diagram of Hereditary and Mutation-Invariant Properties, Represented as Classes of Quivers}
    \label{fig-venn-diagram}
\end{figure}

\subsection{Acknowledgements}

Part of this work was presented at the AMS 2023 Spring Central Sectional meeting.
I would like to thank Eric Bucher, John Machacek, and Nicholas Ovenhouse for inviting me.
Also, I would like to thank Scott Neville for his discussion about this paper and how it relates to his paper with Sergey Fomin \cite{fomin_long_2023}.
Finally, I would like to thank my advisor, Kyungyong Lee, for his guidance, support, and help revising this paper.

\section{Preliminaries} \label{sec-preliminaries}

We begin with some preliminary information.

\begin{Def} \label{def-quivers}
    A quiver $Q$ is defined as a tuple $(Q_0, Q_1, s,t)$, where $Q_0$ is a set of vertices, $Q_1$ is a set of arrows, and maps $s,t: Q_1 \to Q_0$ taking any arrow $a \in Q_1$ to its starting vertex $s(a)$ and its terminal vertex $t(a)$, i.e., $s(a) \to t(a)$ is the arrow $a$ in $Q$.
    We further restrict our definition so that $s(a) \neq t(a)$ (no loops) for all $a \in Q_1$ and that there is no pair of arrows $a,b \in Q_1$ where $s(a) = t(b)$ and $t(a) = s(b)$ (no 2-cycles).
    This naturally forces all arrows between two vertices to point in the same direction.
    For example, the following are two distinct quivers:
    \[\begin{tikzcd}
    i & j 
    \arrow[from=1-1,to=1-2]
    \end{tikzcd}
    \text{ and }\begin{tikzcd}
    i & j 
    \arrow[from=1-2,to=1-1]
    \end{tikzcd}\]
    Additionally, whenever we have multiple arrows between the same two vertices, we write the multiplicity, say $m$, above or below the arrow, i.e,
    \[\begin{tikzcd}
    i & j 
    \arrow[from=1-1,to=1-2, "m"]
    \end{tikzcd}\]
    There are several terms that we will use when referring to quivers.
    
    \begin{itemize}
        \item A vertex that only has arrows leaving (entering) is referred to as a \textit{source} (\textit{sink}).

        \item A quiver is called \textit{acyclic} if there exists no (nonempty) directed path of arrows beginning and ending at the same vertex.

        \item A quiver is said to be \textit{abundant} (or \textit{2-complete} \cite{felikson_acyclic_2018} or to have \textit{large weights} \cite{fomin_long_2023}) if there are at least two arrows between every pair of vertices. 

        \item A quiver is \textit{connected} if there exists a sequence of edges (considered without their directions) connecting any two pairs of vertices.

        \item A quiver $Q = (Q_0, Q_1, s,t)$ is called a \textit{full subquiver} (often called an \textit{induced subquiver}) of a quiver $P = (P_0,P_1,s',t')$ if $Q_0 \subset P_0$, if $Q_1 \subseteq P_1$, if the maps $s$ and $t$ are restrictions of $s'$ and $t'$ to the set $Q_1$, and if $s(a), t(a) \in Q_0$ implies that $a \in Q_1$ for any arrow $a \in P_1$.
        For ease of use, the notation $P \setminus V$ will refer to the full subquiver induced by the vertex set $P_0 \setminus V$ for some set $V$.
        A quiver property is called \textit{hereditary} if it is preserved by restriction to any full subquiver.
    \end{itemize}
    Furthermore, we finish with a small matter of notation.
    Fix a quiver, denoted by a capital letter, say $Q$.
    Then the number of arrows between vertex $i$ and $j$ in $Q$ is denoted by the lowercase letter indexed by the vertices, $q_{ij}$, which is positive if $i \to j$ and negative if $j \to i$.
\end{Def}

\begin{Def} \label{def-mutation}
    \textit{Quiver mutation} is defined for every quiver $Q$ and vertex $v$ in $Q_0$ by the following:
    \begin{itemize}
        \item For every path of the form $a \to v \to b$ in $Q$, add an arrow from $a$ to $b$;
        
        \item Flip the direction of every arrow touching $v$;
        
        \item Finally, if any 2-cycles were formed (paths of length two that begin and end at the same vertex), remove them.
    \end{itemize}
   This process will always mutate a quiver into another quiver, and we denote the mutation of $Q$ at vertex $v$ by $\mu_v(Q)$.
   If we have a sequence of vertices, $\bw = [i_1,i_2,i_3,\dots,i_n]$, then we use 
   $$\mu_\bw(Q) = \mu_{i_n}(\mu_{[i_1,i_2,\dots,i_{n-1}]}(Q))$$
   to denote the result of mutating $Q$ at each vertex $i_j$ in order.
   We call $\bw$ a \textit{mutation sequence}.
   Additionally, we assume that all of our mutation sequences are \textit{reduced}, i.e., we never see mutation at a vertex and then another immediate mutation at that same vertex.
   Mutation is an involution, so two mutations at the same vertex can be safely replaced by not mutating in the first place.
\end{Def}

\begin{Def} \label{def-mutation-class}
    Two quivers are said to be \textit{mutation-equivalent} if there exists a sequence of mutations taking one to the other.
    The collection of all quivers mutation-equivalent to a quiver $Q$ is denoted $[Q]$, and it is called the \textit{labelled mutation class} of $Q$.
    If $[Q]$ is a finite (infinite) set, then $Q$ is said to be \textit{mutation-finite} (\textit{mutation-infinite}).
    If $[Q]$ has an acyclic quiver in it, then $Q$ is said to be \textit{mutation-acyclic}.
    Otherwise, it is said to be \textit{mutation-cyclic}.
    A property that is invariant under mutation, called a \textit{mutation-invariant}, can be considered as a property of the labelled mutation class.
    Additionally, we can construct the \textit{labelled mutation graph} of $Q$ by letting $[Q]$ act as the vertex set and forming an edge between two quivers whenever they are a single mutation apart, where this edge is labelled by the mutation.
\end{Def}

\begin{Rmk} \label{rmk-special-cycles}
    It is well-known that if $i \to j$ is a single arrow in a quiver $Q$, then $P = \mu_{[i,j,i,j,i]}(Q) = \mu_{[j,i,j,i,j]}(Q)$ is a copy of $Q$ found by swapping vertices $i$ and $j$ (for example, see \cite{lee_positivity_2015}).
    Additionally, if there is no arrow connecting $i$ and $j$ in $Q$, then $Q = \mu_{[i,j,i,j]}(Q) = \mu_{[j,i,j,i]}(Q)$ (for example, see \cite{fomin_introduction_2021}).
\end{Rmk}

We now may present the concept of forks and the relevant results of Warkentin's dissertation.

\begin{Def}{\cite[Definition 2.1]{warkentin_exchange_2014}} \label{def-forks}
A \textit{fork} is an abundant quiver $F$, where $F$ is not acyclic and where there exists a vertex $r$, called the point of return, such that
\begin{itemize}
    \item For all $i \in F^{-}(r)$ and $j \in F^{+}(r)$ we have $f_{ji} > f_{ir}$ and $f_{ji} > f_{rj}$, where $F^{-}(r)$ is the set of vertices with arrows pointing towards $r$ and $F^{+}(r)$ is the set of vertices with arrows coming from $r$.
    
    \item The full subquivers induced by $F^{-}(r)$ and $F^{+}(r)$ are acyclic.
\end{itemize}
An example of a fork is given by 
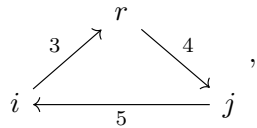
\[\begin{tikzcd}
    & r \\
    i & & j
    \arrow[from=2-1, to=1-2, "3"]
    \arrow[from=1-2, to=2-3, "4"]
    \arrow[from=2-3, to=2-1, "5"]
\end{tikzcd},\]
where $r$ is the point of return.
\end{Def}

\begin{Lem}{\textup{\cite[Lemma 2.4]{warkentin_exchange_2014}}} \label{lem-fork-acyclic}
    Let $F$ be a fork with point of return $r$.
    Then a full subquiver $Q$ of $F$ is either abundant acyclic or a fork with the same point of return $r$.
\end{Lem}

\begin{Lem}{\textup{\cite[Lemma 2.5]{warkentin_exchange_2014}}} \label{lem-fork-mutation}
    Let $F$ be either a fork with point of return $r$ or an abundant acyclic quiver.
    If $k$ is a vertex in $F$ that is not the point of return, not a source, or not a sink, then $\mu_k(F)$ is a fork with point of return $k$.
    Additionally, the total number of arrows contained in $\mu_k(F)$ is strictly greater than the total number of arrows in $F$. 
\end{Lem}

\begin{Lem}{\textup{\cite[Lemma 2.8]{warkentin_exchange_2014}}} \label{lem-fork-tree}
    Let $\Gamma$ be the labelled mutation graph of a quiver with $n$ vertices and assume that one vertex of $\Gamma$ is a fork $F$ with point of return $r$.
    Then the edge $e$ is a bridge between $F$ and $\mu_r(F)$, and the connected component $\Gamma'$ of $\Gamma \setminus e$ containing $F$ is a $n$-regular tree except at the quiver $F$, where only $n-1$ edges appear and every vertex of the tree is itself a fork.
\end{Lem}

\begin{Thm}{\textup{\cite[Theorem 3.2]{warkentin_exchange_2014}}} \label{thm-connected-fork}
    A connected quiver is mutation-infinite if and only if it is mutation-equivalent to a fork.
\end{Thm}

\begin{Cor} \label{cor-connected-fork}
    A disconnected quiver is not mutation-equivalent to any forks.
\end{Cor}

We shall use all of these results in our proof of the main result.
Before we do so, we must first discuss abundant acyclic quivers. 

\subsection{Some Properties of Abundant Acyclic Quivers} \label{sec-abundant-acyclic}

\begin{Def}{\textup{\cite{bang-jensen_classes_2018}}} \label{def-acyclic-ordering}
    Let $Q$ be a quiver on $n$ vertices.
    Then we say that an ordering $v_1 \prec v_2 \prec \dots \prec v_n$ on the vertices is acyclic if whenever there exists an arrow $v_i \to v_j$ then $v_i \prec v_j$.
\end{Def}

\begin{Prop}{\textup{\cite[Proposition 3.1.1]{bang-jensen_classes_2018}}} \label{prop-acyclic-source-sink}
    Let $Q$ be an acyclic quiver.
    Then $Q$ has at least one source and at least one sink.
\end{Prop}

\begin{Prop}{\textup{\cite[Proposition 3.1.2]{bang-jensen_classes_2018}}} \label{prop-acyclic-ordering}
    Let $Q$ be an acyclic quiver.
    Then there exists an acyclic ordering of its vertices.
\end{Prop}

Since abundant acyclic quivers have an arrow between every pair of vertices, the acyclic ordering produced is unique.
This gives the following corollary.

\begin{Cor} \label{cor-acyclic-ordering}
    Let $Q$ be an abundant acyclic quiver on $n$ vertices.
    Then there exists a unique acyclic ordering of its vertices.
    If $v_1 \prec v_2 \prec \dots \prec v_n$ is the acyclic ordering, then $v_1$ is the only source and $v_n$ is the only sink.
\end{Cor}

This corollary will allow us to prove a useful result on the number of distinct abundant acyclic quivers in the labelled mutation class of an abundant acyclic quiver.
The lemma may be well known in another form to others, but we include the proof for the sake of completeness.

\begin{Lem} \label{lem-abundant-acyclic-count}
    Let $Q$ be an abundant acyclic quiver on $n$ vertices.
    Then the labelled mutation class of $Q$ contains exactly $n$ non-forks.
    Additionally, mutating $Q$ at sources (or sinks) repeatedly produces a $n$-cycle in the labelled mutation graph of $Q$.
\end{Lem}

\begin{proof}
    As $Q$ is an abundant acyclic quiver, mutating at any vertex other than a source or a sink will produce a fork by Lemma \ref{lem-fork-mutation}.
    Further mutation will only produce forks by Lemma \ref{lem-fork-tree}.
    Let $v_1 \prec v_2 \prec \dots \prec v_n$ be an acyclic ordering created from $Q$ by Proposition \ref{prop-acyclic-ordering}.
    Then Corollary \ref{cor-acyclic-ordering} tells us that the acyclic ordering is unique, that $v_1$ is the only source, and that $v_n$ is the only sink.
    Then mutation at $v_1$ will produce another abundant acyclic quiver $Q' = \mu_{v_1}(Q)$.
    This forces the acyclic ordering produced from $Q'$ to be $v_2 \prec \dots \prec v_n \prec v_1$, where $v_2$ is now the only source and $v_1$ the only sink.
    Note that mutating at the sink would only produce $Q$ and mutating at any vertex other than a source or a sink would produce a fork.
    We may then repeat the process of mutating at sources a further $n-1$ times, arriving back at our original quiver $Q$.
    If we were instead to mutate at sinks only, we would traverse the $n$-cycle formed in the labelled mutation graph in the reverse direction as to mutating at sources only.
    Therefore, the labelled mutation class of $Q$ contains exactly $n$ non-forks.
\end{proof}

If we were to look at the labelled mutation graph of an abundant acyclic quiver on $n$ vertices, then it would look something like the following:
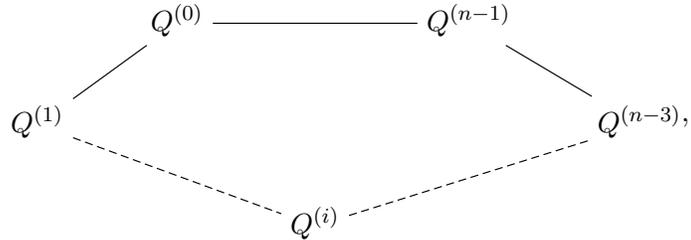
\[\begin{tikzcd}[cramped]
&  Q^{(0)} & & Q^{(n-1)} \\
Q^{(1)} & & & & Q^{(n-3)}\\
& &  Q^{(i)}
\arrow[from=1-2, to=1-4, no head]
\arrow[from=1-2, to=2-1, no head]
\arrow[from=1-4, to=2-5, no head]
\arrow[from=2-1, to=3-3, no head, dashed]
\arrow[from=3-3, to=2-5, no head, dashed]
\end{tikzcd},\]
where each of the $Q^{(i)}$ is an abundant acyclic quiver, each of the solid edges represents mutation at a source or a sink, and mutating anywhere else produces a fork.

\section{Finite Forkless Part} \label{sec-forkless-part}

We begin with our main definitions.

\begin{Def} \label{def-forkless-part}
    Let $Q$ be a quiver.
    Then we define the \textit{forkless part} of $Q$ as the induced subgraph of the labelled mutation graph of $Q$ that contains all non-forks mutation-equivalent to $Q$ \cite{warkentin_exchange_2014}.
    If the forkless part of $Q$ is a finite set, we say that $Q$ has a \textit{finite forkless part}.
    Furthermore, if the forkless part has $m$ quivers in it, we say that $Q$ has a $m$ forkless part.
\end{Def}

\begin{Rmk} \label{rmk-forkless}
    Note that we say \textit{the} forkless part of the labelled mutation graph of $Q$.
    Lemma \ref{lem-fork-tree} demonstrates that only one such subgraph will exist and that this subgraph is also a convex subgraph, as any path between quivers in the forkless part need not pass through a fork.
\end{Rmk}

Our Lemma \ref{lem-abundant-acyclic-count} can be restated as the following corollary.

\begin{Cor} \label{cor-abundant-acyclic}
    Let $Q$ be an abundant acyclic quiver on $n$ vertices.
    Then $Q$ has a $n$ forkless part.
\end{Cor}

Another corollary follows for forks in general, as all subquivers of forks are either abundant acyclic or are forks that share a point of return by Lemma \ref{lem-fork-acyclic}.

\begin{Cor} \label{lem-fork-subquiver}
    Let $Q$ be a fork with point of return $r$.
    Then any full subquiver $Q'$ --- necessarily a connected subquiver --- that does not have $r$ as a vertex has a finite forkless part.
    Indeed, the subquiver $Q'$ is mutation-equivalent to exactly $m$ distinct abundant acyclic quivers, where $m$ is the number of vertices of $Q'$.
\end{Cor}

Our main theorem may then be proved.

\begin{Thm} \label{thm-forkless-quiver}
    Let $Q$ be a quiver that has a finite forkless part.
    Then any (possibly disconnected) full subquiver $Q'$ also has a finite forkless part.
    Furthermore, if $M$ is the number of non-forks mutation-equivalent to $Q$ and $m$ is the number of vertices in $Q'$, then the number of non-forks mutation-equivalent to $Q'$ is bounded by $\max\{M, m\}$.
\end{Thm}

\begin{proof}
    As $Q'$ is a full subquiver of $Q$, every quiver that is mutation-equivalent to $Q'$ is a full subquiver of a quiver mutation-equivalent to $Q$.
    If $Q'$ is not mutation-equivalent to a full subquiver of a fork mutation-equivalent to $Q$, then $Q'$ is only mutation-equivalent to subquivers of non-forks mutation-equivalent to $Q$.
    Then $Q'$ must have a finite forkless part, as there are only $M$ many such quivers that are mutation-equivalent to $Q$.

    We may then assume that $Q'$ is mutation-equivalent to a full subquiver of a fork mutation-equivalent to $Q$.
    Then $Q'$ is mutation-equivalent to either an abundant acyclic quiver or a fork by Lemma \ref{lem-fork-subquiver}.
    If $Q'$ is mutation-equivalent to an abundant acyclic quiver, then it has finite forkless part with $m$ non-forks in it by Corollary \ref{cor-abundant-acyclic}.
    We assume then that $Q'$ is not mutation-equivalent to any abundant acyclic quiver.
    If $Q'$ is mutation-equivalent only to subquivers of a fork mutation-equivalent to $Q$ that are themselves forks, then this forces $Q'$ to be mutation-equivalent to at most $M$ non-forks, as there are only $M$ possible non-forks that it can be mutation-equivalent to a subquiver of.
    Therefore, any full subquiver $Q'$ of a quiver that has a finite forkless part also has a finite forkless part with the number of non-forks bounded by $\max\{M,m\}$.
\end{proof}

Since having a finite forkless part is a property of the labelled mutation class, we have the following corollary as the main result of our paper.

\begin{Cor} \label{cor-forkless-hereditary}
    Having a finite forkless part is a hereditary and mutation-invariant property.
\end{Cor}

Additionally, we have a nice consequence of Theorem \ref{thm-forkless-quiver} concerning disconnected full subquivers, as disconnected quivers are not mutation-equivalent to forks.

\begin{Cor} \label{cor-disconnected-finite}
    Let $Q$ be a quiver with a finite forkless part.
    If $Q'$ is a disconnected full subquiver of $Q$, then $Q'$ must be mutation-finite.
\end{Cor}

There is no lower bound on the number of non-forks in the forkless part of a quiver \cite[Corollary 4.4]{warkentin_exchange_2014}.
For example, the quiver
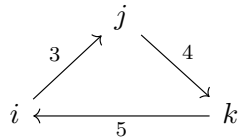
\[\begin{tikzcd}
    & j \\
    i & & k
    \arrow[from=2-1, to=1-2, "3"]
    \arrow[from=1-2, to=2-3, "4"]
    \arrow[from=2-3, to=2-1, "5"]
\end{tikzcd}\]
is mutation-equivalent to zero non-forks (mutating at $j$ produces a fork with point of return $j$), meaning that it has 0 forkless part.
Furthermore, it is mutation-cyclic by the same reasoning.

An upper bound on the number of non-forks in the finite forkless part of a quiver does not exist, at least in the case of quivers with four vertices.
For example, the quiver
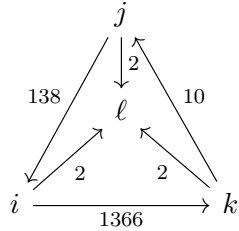
\[\begin{tikzcd}
& j\\
& \ell\\
i & & k
\arrow[from=3-3, to=1-2, "10"']
\arrow[from=3-1, to=3-3, "1366"']
\arrow[from=1-2, to=3-1, "138"']
\arrow[from=1-2, to=2-2, "2"]
\arrow[from=3-3, to=2-2, "2"]
\arrow[from=3-1, to=2-2, "2"']
\end{tikzcd}\]
has a finite forkless part with 14 elements. 
Other quivers with a large forkless part can be constructed in a similar manner by attaching a sink in the manner above to a quiver mutation-equivalent to an abundant acyclic quiver, as detailed in Sergey Fomin and Scott Neville's upcoming paper \cite{fomin_long_2023}. 

Additionally, if we instead attached a sink to a mutation-cyclic quiver on 3 vertices, then we arrive at a quiver like
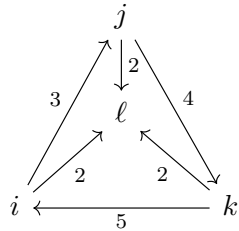
\[\begin{tikzcd}
& j\\
& \ell\\
i & & k
\arrow[from=1-2, to=3-3, "4"]
\arrow[from=3-3, to=3-1, "5"]
\arrow[from=3-1, to=1-2, "3"]
\arrow[from=1-2, to=2-2, "2"]
\arrow[from=3-3, to=2-2, "2"]
\arrow[from=3-1, to=2-2, "2"']
\end{tikzcd}\]
that is mutation-equivalent to exactly two non-forks (mutate at $\ell$ to produce the only other non-fork in its labelled mutation class).
Furthermore, every mutation-finite quiver is readily seen to have a finite forkless part, and every abundant acyclic quiver is known to have a finite forkless part.
This shows that the class of quivers with a finite forkless part is broad.

However, the class of quivers with an infinite forkless part is equally broad.
For example, the quiver 
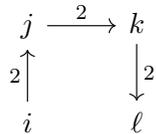
\[\begin{tikzcd}
j & k \\
i & \ell 
\arrow[from=2-1, to=1-1, "2"]
\arrow[from=1-1, to=1-2, "2"]
\arrow[from=1-2, to=2-2, "2"]
\end{tikzcd}\]
is mutation-equivalent to an infinite number of non-forks (alternate mutations at vertices $k$ and $\ell$), but it is clearly mutation-acyclic. 
In a forthcoming paper, we will provide a generalization of this example. 

In general, quivers can have disconnected full subquivers, and this brings many examples of quivers without a finite forkless part.
For example, the quiver
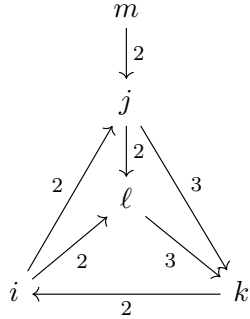
\[\begin{tikzcd}
& m \\
& j\\
& \ell\\
i & & k
\arrow[from=1-2, to=2-2, "2"]
\arrow[from=2-2, to=3-2, "2"]
\arrow[from=2-2, to=4-3, "3"]
\arrow[from=3-2, to=4-3, "3"']
\arrow[from=4-1, to=2-2, "2"]
\arrow[from=4-1, to=3-2, "2"']
\arrow[from=4-3, to=4-1, "2"]
\end{tikzcd}\]
has a disconnected full subquiver formed by removing the vertex $j$
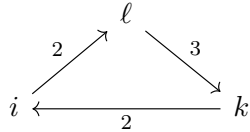
\[\begin{tikzcd}
& m \\
&  \\
& \ell\\
i & & k
\arrow[from=3-2, to=4-3, "3"]
\arrow[from=4-1, to=3-2, "2"]
\arrow[from=4-3, to=4-1, "2"]
\end{tikzcd}\]
As the subquiver induced by $\{i,\ell,k\}$ is mutation-infinite and the full subquiver induced by $\{i, \ell, j, k\}$ is disconnected, our original quiver does not have a finite forkless part by Corollary \ref{cor-disconnected-finite}.

\section{Keys, Pre-forks, And Their Mutations} \label{sec-keys-pre-forks}

Before we continue on to our next hereditary and mutation-invariant property, we must first define a small generalization of abundant acyclic quivers and study their mutations.

\begin{Def} \label{def-keys}
    Let $Q$ be a quiver.
    We say that $Q$ is a \textit{key} with vertices $\{k,k'\}$ if it satisfies the following:
    \begin{itemize}
        \item Both full subquivers $Q \setminus \{k\}$ and $Q \setminus \{k'\}$ are abundant acyclic quivers.

        \item Let $i \in Q_0 \setminus \{k,k'\}$.
        Then we have either
        
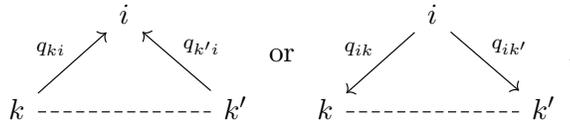
\[\begin{tikzcd}
            & i \\
            k & & k' 
            \arrow[from=2-1, to=1-2, "q_{ki}"]
            \arrow[from=2-3, to=1-2, "q_{k'i}"']
            \arrow[from=2-1, to=2-3, no head, dashed]
        \end{tikzcd}
        \text{ or }
        \begin{tikzcd}
            & i \\
            k & & k' 
            \arrow[from=1-2, to=2-1, "q_{ik}"']
            \arrow[from=1-2, to=2-3, "q_{ik'}"]
            \arrow[from=2-1, to=2-3, no head, dashed]
        \end{tikzcd},\]
        where the dashed line represents any number of arrows, including zero.
    \end{itemize}
\end{Def}

The definition of a key gives us an immediate corollary on some of its properties.

\begin{Cor} \label{cor-key-subquiver}
    Let $Q$ be a key.
    Then there is no directed path of length greater than 1 from $k$ to $k'$ or from $k'$ to $k$, forcing $Q$ to be an acyclic quiver.
    Additionally, if $P$ is a full subquiver of $Q$, then it is either abundant acyclic or itself a key.
\end{Cor}

\begin{proof}
    Suppose that there exists a cycle of minimum length in $Q$.
    Then this cycle must pass through both $k$ and $k'$, as $Q \setminus \{k\}$ and $Q \setminus \{k'\}$ are abundant acyclic by the definition of a key.
    Then there is an directed path from $k$ to $k'$.
    Let $j$ be the final vertex of this path such that $j \to k'$.
    If $j \neq k$, then we get that $j \to k$ by the second property of keys, providing us a cycle contained in $Q \setminus \{k'\}$.
    This creates a contradiction, and if $j = k$, then we can use the same argument on the directed path from $k'$ to $k$.
    This proves that $Q$ must be acyclic.
\end{proof}

Knowing the set $Q_{k'}^k := [Q^+(k) \cap Q^-(k')] \cup [Q^+(k') \cap Q^-(k)]$ is useful for any quiver $Q$ when mutating.
For example, the second condition of a key implies that $Q_{k'}^k = \emptyset$, which means that mutation at any vertex other than $k$ or $k'$ will not modify the number or direction of arrows between $k$ and $k'$.
Warkentin uses the set $Q_{k'}^k$ judiciously when constructing future types of quivers.
We must first prove some results on our newly defined keys, similar to the ones we showed for abundant acyclic quivers.

\begin{Lem} \label{lem-key-ordering}
    Let $Q$ be a key with vertices $\{k,k'\}$.
    If $q_{kk'} = 0$, then $Q$ has at most two possible acyclic orderings on its vertices, with $k \prec k'$ in one and $k' \prec k$ in the other.
    If $|q_{kk'}| \geq 1$, then $Q$ has a unique acyclic ordering on its vertices.
\end{Lem}

\begin{proof}
    From the definition of a key, we have that $k \to i$ if and only if $k' \to i$ for any vertex $i \neq k,k'$.
    When translated into the language of an acyclic ordering, we have that $k \prec i$ if and only if $k' \prec i$ for $i \neq k,k'$.
    Since both $Q \setminus \{k\}$ and $Q \setminus \{k'\}$ are abundant acyclic quivers, there is a unique acyclic ordering of vertices for each quiver.
    It follows that $i \prec j$ in the first acyclic ordering if and only if $i \prec j$ in the second for all vertices $i,j$ such that $\{i,j\} \cap \{k,k'\} \neq \{k,k'\}$.
    We then need only worry about the relation between $k$ and $k'$.
    
    If $q_{kk'} = 0$, then we can have either $k \prec k'$ or $k' \prec k$, producing our two promised acyclic orderings.
    If $|q_{kk'}| \geq 1$, then we have exclusively $k \prec k'$ or $k' \prec k$, depending on the direction of the arrow between the two vertices.
\end{proof}

\begin{Cor} \label{cor-key-source-sink}
    Let $Q$ be a key with vertices $\{k,k'\}$ and $v \neq k,k'$ be any other vertex.
    Then $k \prec v$ if and only if $k' \prec v$.
    If $q_{kk'} = 0$, then $Q$ has either one source and one sink, two sources and one sink, or one source and two sinks.
    In the latter two cases, both $k$ and $k'$ are a source or a sink respectively.
    If $|q_{kk'}| =1 $, then $Q$ has one source and one sink.
\end{Cor}

Notice how every key with $|q_{kk'}| \geq 2$ is necessarily an abundant acyclic quiver.
Indeed, every abundant acyclic quiver can be thought of as a key.
Just take two neighboring elements in the acyclic ordering as $k$ and $k'$.
This idea of generalizing abundant acyclic quivers to keys can be applied to forks as well, giving us pre-forks.

\begin{Def}{\textup{\cite[Definition 3.7]{warkentin_exchange_2014}}} \label{def-pre-fork}
A pre-fork is a quiver $P$ with two vertices $k \neq k'$ such that the following holds:
\begin{enumerate}
    \item Both full subquivers $P \setminus \{k\}$ and $P \setminus \{k'\}$ are forks with common point of return $r$.

    \item Let $i \neq k,k' \in P_0$.
        Then we have either
        
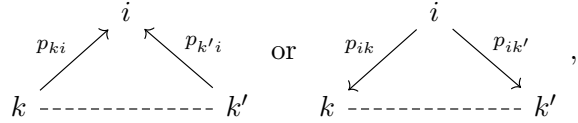
\[\begin{tikzcd}
            & i \\
            k & & k' 
            \arrow[from=2-1, to=1-2, "p_{ki}"]
            \arrow[from=2-3, to=1-2, "p_{k'i}"']
            \arrow[from=2-1, to=2-3, no head, dashed]
        \end{tikzcd}
        \text{ or }
        \begin{tikzcd}
            & i \\
            k & & k' 
            \arrow[from=1-2, to=2-1, "p_{ik}"']
            \arrow[from=1-2, to=2-3, "p_{ik'}"]
            \arrow[from=2-1, to=2-3, no head, dashed]
        \end{tikzcd},\]
        where the dashed line represents any number of arrows, including zero.
\end{enumerate}
We usually denote a pre-fork $P$ with its triple of vertices $\{r,k,k'\}$, where $r$ is referred to as the point of return. 
\end{Def}

The reason for the name ``pre-forks`` comes from the fact that nearly all pre-forks are forks.
It is only when $|p_{kk'}| < 2$ that they differ \cite[Lemma 6.1]{warkentin_exchange_2014}.
Thus pre-forks are related to forks in the same way that keys are related to abundant acyclic quivers.
Additionally, just like forks, they have some nice mutation properties.

\begin{Lem}{\textup{\cite[Corollary 6.2]{warkentin_exchange_2014}}} \label{lem-pre-fork-mutation}
    Let $P$ be a pre-fork with vertices $\{r,k,k'\}$.
    Then $Q = \mu_m(P)$ is a pre-fork with vertices $\{m,k,k'\}$ for all vertices $m \notin \{r,k,k'\}$.
    Additionally, we have that $|Q_1| > |P_1|$ and that $p_{kk'} = q_{kk'}$.
\end{Lem}

\begin{Lem}{\cite[Lemma 3.8]{warkentin_exchange_2014}} \label{lem-pre-fork-mutation-k}
    Let $P$ be a pre-fork with vertices $\{r,k,k'\}$, let $Q = \mu_k(P)$, and let $Q' = \mu_{k'}(P)$.
    Then
    \begin{itemize}
        \item $|Q_1| > |P_1|$ and $|Q_1'| > |P_1|$;

        \item $Q \setminus \{k'\}$ and $Q' \setminus \{k\}$ are forks with respective points of return $k$ and $k'$;

        \item $Q \setminus \{k\}$ and $Q' \setminus \{k'\}$ are abundant acyclic quivers;

        \item $Q_{k'}^k = (Q')_{k'}^k = P_0 \setminus \{k,k'\}$ as sets;

        \item $\mu_h(Q)$ and $\mu_h(Q')$ are both forks with point of return $h$ for all $h \neq k,k'$.
    \end{itemize} 
\end{Lem}

Warkentin was able to show analogues of Lemma \ref{lem-fork-tree} for pre-forks.
In essence, this means that the quivers found after mutating through a pre-fork are well-controlled enough that we can ignore them, eventually developing the idea of pre-forkless part.
As such, it will be advantageous to explore the relation between pre-forks and keys, much like the relation between forks and abundant acyclic quivers.
We begin by exploring the subquivers of pre-forks.

\begin{Lem} \label{lem-pre-fork-key-subquiver}
    Let $P$ be a pre-fork with vertices $\{r,k,k'\}$.
    Then $P \setminus \{r\}$ is a key with vertices $\{k,k'\}$.
\end{Lem}

\begin{proof}
    As $P$ is a pre-fork with vertices $\{r,k,k'\}$, we know that the full subquivers $P \setminus \{k,r\}$ and $P \setminus \{k',r\}$ are both abundant acyclic from Lemma \ref{lem-fork-acyclic}.
    Additionally, we must have that $(P \setminus \{r\})_{k'}^k = \emptyset$ from $P_{k'}^k = \emptyset$.
    Therefore, the full subquiver $P \setminus \{r\} $ must be a key with vertices $\{k,k'\}$.
\end{proof}

\begin{Lem} \label{lem-pre-fork-subquivers}
    Let $P$ be a pre-fork with vertices $\{r,k,k'\}$.
    Then any full subquiver of $P$ is one of the following: 
    \begin{itemize}
        \item abundant acyclic;

        \item a key with vertices $\{k,k'\}$;

        \item a fork with point of return $r$;

        \item a pre-fork with vertices $\{r,k,k'\}$.
    \end{itemize}
\end{Lem}

\begin{proof}
    Both full subquivers of $P$ without $k$ and full subquivers of $P$ without $k'$ are subquivers of forks by Definition \ref{def-pre-fork}.
    Thus they are either abundant acyclic or forks with point of return $r$ by Lemma \ref{lem-fork-acyclic}.
    Additionally, we have that $P \setminus \{r\}$ is a key with vertices $\{k,k'\}$ by Lemma \ref{lem-pre-fork-key-subquiver}.
    Corollary \ref{cor-key-subquiver} tells us that every subquiver must also be abundant acyclic or a key.
    We are then left with full subquivers of $P$ that contain the vertices $r, k$, and $k'$.
    Let $Q$ be such a quiver.
    If $Q$ is acyclic, then $Q \setminus \{k\}$ and $Q \setminus \{k'\}$ are abundant acyclic with $Q_{k'}^k = \emptyset$, making it a key.
    As such, we will reduce to the case where $Q$ has a cycle.

    If this cycle is contained in $P \setminus \{k,k'\}$, then the cycle will be contained in both forks $P \setminus \{k\}$ and $P \setminus \{k'\}$.
    This forces $Q \setminus \{k\}$ and $Q \setminus \{k'\}$ to be forks with point of return $r$ by Lemma \ref{lem-fork-acyclic}, making $Q$ a pre-fork with vertices $\{r,k,k'\}$.
    The only problem arises when no cycle is contained in $P \setminus \{k,k'\}$.
    Assume without loss of generality that our cycle passes through $k$.
    We will show that this produces a cycle in both $Q \setminus \{k\}$ and $Q \setminus \{k'\}$.

    If there exists an arrow $r \to k$ ($k \to r$), then there exists an arrow $r \to k'$ ($k' \to r$), as $Q_{k'}^k = \emptyset$ and both $Q \setminus \{k\}$ and $Q \setminus \{k'\}$ are abundant.
    As $Q \setminus \{r\}$ is an acyclic quiver, we must have that any cycle passing through $k$ also passes through $r$. 
    The existence of a cycle passing through $k$ and $r$ implies that there is a path of arrows from $k$ to $r$ (from $r$ to $k$).
    As $Q_{k'}^k = \emptyset$, this tells us that there is an identical path of arrows from $k'$ to $r$ ($r$ to $k'$).
    This produces two cycles passing through $k$ and $k'$ respectively.
    Now we must show that these two cycles do not need to pass through both $k$ and $k'$.
    
    From $Q_{k'}^k = \emptyset$, a cycle passing through $k$ --- if it passes through $k'$ --- can bypass $k'$ by taking the same path through $k$ instead.
    The same can be said about any cycle passing through $k'$.
    We have then produced two cycles: one passing through $k$ but not $k'$ and one through $k'$ but not $k$.
    They are then respectively contained in $Q \setminus \{k\}$ and $Q \setminus \{k'\}$.
    Therefore, the quiver $Q$ must be a pre-fork with vertices $\{r,k,k'\}$.
\end{proof}

Once we see that subquivers of pre-forks and keys are well controlled, it becomes advantageous to understand how keys act under mutation.
As such, Lemma \ref{lem-pre-fork-mutation} and our new Lemma \ref{lem-key-mutation} combine to give us an analogue of Lemma \ref{lem-fork-mutation}.

\begin{Lem} \label{lem-key-mutation}
    Let $Q$ be a key with vertices $\{k,k'\}$.
    Then the quiver $P = \mu_r(Q)$ is a pre-fork with vertices $\{r,k,k'\}$, $|P_1| > |Q_1|$, and $p_{k,k'} = q_{k,k'}$ if $r \neq k,k'$ and $r$ is not a sink or source in $Q$.
    Furthermore, if $r \neq k,k'$ is a sink or source in $Q$, then $P$ is a key with vertices $\{k,k'\}$, $|P_1| = |Q_1|$, and $p_{k,k'} = q_{k,k'}$.
\end{Lem}

\begin{proof}.
    Since $Q$ is acyclic, there are no 3-cycles contained in $Q$.
    If $r$ is not a source or a sink, this shows that a positive number of arrows will be added to $Q$ and that only arrows directly touching $r$ will reverse direction.
    Thus $P_{k'}^k = \emptyset$, $|P_1| > |Q_1|$, and $p_{k,k'} = q_{k,k'}$ all from $Q_{k'}^k = \emptyset$.

    As $Q \setminus \{k\}$ and $Q \setminus \{k'\}$ are abundant acyclic quivers, we know that $P \setminus \{k\}$ and $P \setminus \{k'\}$ will be forks with point of return $r$.
    Hence, the quiver $P = \mu_r(Q)$ must be a pre-fork with vertices $\{r,k,k'\}$ as long as $r \neq k,k'$ and $r$ is not a sink or source in $Q$.

    If $r \neq k,k'$ is a sink or source in $Q$, then mutation at $r$ will only affect the direction of arrows touching $r$ and not the number of arrows in the quiver.
    Then $Q_{k'}^k = \emptyset$ implies that $P_{k'}^k = \emptyset$.
    Additionally, mutating at a sink or source will not introduce any cycles.
    Therefore, if $r \neq k,k'$ is a sink or source in $Q$, then $P$ is a key with vertices $\{k,k'\}$, $|P_1| = |Q_1|$, and $p_{k,k'} = q_{k,k'}$.
\end{proof}

\begin{Lem} \label{lem-key-mutation-k}
    Let $Q$ be a key with vertices $\{k,k'\}$ where neither is a sink or a source, let $P = \mu_k(Q)$, and let $P' = \mu_{k'}(Q)$.
    Then
    \begin{itemize}
        \item $|P_1| > |Q_1|$ and $|P_1'| > |Q_1|$;

        \item $P \setminus \{k'\}$ and $P' \setminus \{k\}$ are forks with respective points of return $k$ and $k'$;

        \item $P \setminus \{k\}$ and $P' \setminus \{k'\}$ are abundant acyclic quivers;

        \item $P_{k'}^k = (P')_{k'}^k = P_0 \setminus \{k,k'\}$ as sets;

        \item and $\mu_r(P)$ and $\mu_r(P')$ are both forks with point of return $r$ for all $r \neq k,k'$.
    \end{itemize} 
\end{Lem}

\begin{proof}
    Both the first and fourth statements follow from $Q_{k'}^k = \emptyset$ and $Q$ being acyclic.
    The second statement follows from Lemma \ref{lem-fork-mutation}, as $Q \setminus \{k\}$ and $Q \setminus \{k'\}$ are both abundant acyclic and neither $k$ or $k'$ are a sink or source.
    We need then only show the third and fifth statements, beginning with the third.
    Without loss of generality, we can restrict to proving both statements for $P$.
    
    The second statement implies that any cycle in $P \setminus \{k\}$ passes through $k'$, as $P \setminus \{k,k'\}$ must be acyclic.
    As such, there must be a path $i \to k' \to j$ in $P$ for vertices $i,j \neq k$.
    This implies that $j \to k \to i$ in $P$ by $P_{k'}^k = P_0 \setminus \{k,k'\}$.
    Hence $ i \to k \to j$ in $Q$, which implies that $i \to j$ in $P$ as $Q$ is acyclic.
    Thus we may bypass $k'$ in our cycle, contradicting $P \setminus \{k,k'\}$ being acyclic.
    Both $P \setminus \{k\}$ and $P' \setminus \{k'\}$ are then abundant acyclic quivers, proving the third statement.

    Finally, the second and third statements imply that $\mu_r(P) \setminus \{k\}$ and $\mu_r(P) \setminus \{k'\}$ are forks with common point of return $r$ by Lemma \ref{lem-fork-mutation} as $r \in P_{k'}^k$.
    We need only show that the 3-cycle formed from $r,k,$ and $k'$ satisfies the first fork condition to prove that $\mu_r(P)$ is a fork.
    However, since either $k \to r \to k'$ or $k' \to r \to k$ in $P$, we know that there are at least $|p_{kr}p_{rk'}| - |p_{kk'}| > |p_{kr}|, |p_{rk'}| \geq 2$ arrows between $k$ and $k'$ in $\mu_r(P)$, proving our result.
    Therefore $\mu_r(P)$ and $\mu_r(P')$ are both forks with point of return $r$ for all $r \neq k,k'$.
\end{proof}

We have an immediate corollary to Lemma \ref{lem-key-mutation-k} from $Q_{k'}^k = \emptyset$, which will be useful later.

\begin{Cor}
    Let $Q$ be a pre-fork with vertices $\{r,k,k'\}$ or a key with vertices $\{k,k'\}$.
    If 
    
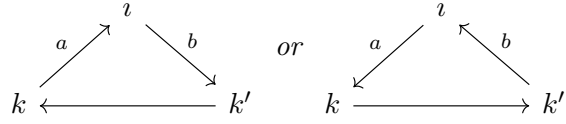
\[\begin{tikzcd}
        & i \\
        k & & k'
        \arrow[from=2-1,to=1-2, "a"]
        \arrow[from=1-2,to=2-3, "b"]
        \arrow[from=2-3,to=2-1]
    \end{tikzcd}
    \text{ or }
    \begin{tikzcd}
        & i \\
        k & & k'
        \arrow[from=1-2,to=2-1, "a"']
        \arrow[from=2-3,to=1-2, "b"']
        \arrow[from=2-1,to=2-3]
    \end{tikzcd}\]
    is a full subquiver of $\mu_k(Q)$, then $b \geq a + 2$.
    If the two quivers are full subquivers of $\mu_{k'}(Q)$, then $a \geq b + 2$.
\end{Cor}

Note that Lemmata \ref{lem-key-mutation} and \ref{lem-key-mutation-k} show that mutation works in essentially the same way for keys as it does for pre-forks.
We need only take care of the case that the vertices $k$ and $k'$ are sinks or sources.
As sinks or sources cannot appear in forks, we do not need a similar lemma for pre-forks.

\begin{Lem} \label{lem-key-k-sink-source-mutation}
    Let $Q$ be a key with vertices $\{k,k'\}$ such that $k$ is a sink or source.
    \begin{itemize}
        \item If $R = \mu_k(Q)$, then $R$ is an acyclic quiver such that $R_{k'}^k = R_0 \setminus \{k,k'\}$ and $\mu_r(R)$ is a fork with point of return $r$ for all $r \neq k,k'$. 

        \item If $P = \mu_{[k,k']}(Q)$, then $P$ is a key with vertices $\{k,k'\}$ and for all $i,j \neq k,k' \in Q_0$ we have that $p_{ik} = -q_{ik}$, that $p_{ik'} = -q_{ik'}$, and that $p_{kk'} = q_{kk'}$.

        \item If $k'$ is not a sink or a source in $Q$, then $R' = \mu_{k'}(Q)$ is a non-acyclic quiver such that $(R')_{k'}^k = (R')_0 \setminus \{k,k'\}$ and $\mu_r(R')$ is a fork with point of return $r$ for all $r \neq k,k'$. 

        \item If $k'$ is not a sink or a source in $Q$, then $P' = \mu_{[k',k]}(Q)$ is a non-acyclic quiver such that $(P')_{k'}^k = (P')_0 \setminus \{k,k'\}$ and $\mu_r(P')$ is a fork with point of return for all $r \neq k,k'$.
    \end{itemize}
    Additionally, for all four of the quivers the number of arrows between vertices $i$ and $j$ for $i,j \neq k,k' \in Q_0$ is fixed.
\end{Lem}

\begin{proof}
    Mutation at $k$ gives us an acyclic quiver $R = \mu_k(Q)$ such that $R_{k'}^k = R_0 \setminus \{k,k'\}$.
    If $r \neq k,k'$, then $\mu_r(R) \setminus \{k\}$ and $\mu_r(R) \setminus \{k'\}$ will both be forks with point of return $r$ by Lemma \ref{lem-fork-mutation}, as $R \setminus \{k\}$ and $R \setminus \{k'\}$ are abundant acyclic quivers.
    Since $r \in R_{k'}^k$ and $k$ was a sink or source in $Q$, the number of arrows between $k$ and $k'$ in $\mu_r(R)$ is greater than both $|q_{k'r}|$ and $|q_{kr}|$.
    As $\mu_r(R) \setminus \{k,r\}$ and $\mu_r(R) \setminus \{k', r\}$ must be abundant acyclic with $r \in R_{k'}^k$, we get that $\mu_r(R)^+(r)$ and $\mu_r(R)^-(r)$ are both acyclic quivers.
    Hence, the quiver $\mu_r(R)$ must be a fork for any $r \neq k,k'$.
    
    From Corollary \ref{cor-key-source-sink}, we know that $k'$ must be the same kind of sink or source in $R$ as $k$ was in $Q$.
    Thus $P = \mu_{[k,k']}(Q)$ is a key with vertices $\{k,k'\}$.
    We also have that $p_{ik} = -q_{ik}$, that $p_{ik'} = -q_{ik'}$, and that $p_{kk'} = q_{kk'}$ for all $i,j \neq k,k' \in Q_0$, as we have only mutated at sinks or sources.

    If $k'$ is neither a sink or a source in $Q$, then we know that $|q_{kk'}| \geq 1$ by Corollary \ref{cor-key-source-sink}.
    Then mutation at $k'$ of $Q$ adds $|q_{k'i}|$ arrows between $k$ and $i$ without changing the direction for every vertex $i \neq k,k'$.
    Thus $R' = \mu_{k'}(Q)$ will be a non-acyclic quiver such that $(R')_{k'}^k = R_0' \setminus \{k,k'\}$.
    Additionally, this implies that $R \setminus \{k\}$ and $R \setminus \{k'\}$ are both abundant acyclic quivers, as no arrows changed direction in $Q \setminus \{k'\}$.
    A similar argument as for $R = \mu_k(Q)$ then shows that $\mu_r(R')$ is a fork with point of return $r$ for all $r \neq k,k'$.

    Since $|r'_{ki}| > |r'_{ik'}|$ for all vertices $i \neq k,k'$, mutating further at $k$ will change the direction of arrows between $k'$ and $i$, adding $|r'_{ki}|$ arrows in the opposite direction.
    Thus $(P')_{k'}^k = P_0 \setminus \{k,k'\}$ for $P' = \mu_{[k',k]}(Q)$.
    Furthermore, we must have that both $P' \setminus \{k'\}$ and $P' \setminus \{k\}$ are abundant acyclic.
    The first is a consequence of $R' \setminus \{k'\}$ being abundant acyclic.
    The second comes from the fact that arrows between $k'$ and $i$ point in the same direction they did as in $Q$, having swapped direction twice while mutating.
    As $Q \setminus \{k\}$ was abundant acyclic, this naturally implies that $P' \setminus \{k\}$ is abundant acyclic.
    Then our previous argument again shows that $\mu_r(P')$ must be a fork with point of return for all $r \neq k,k'$.
    Finally, since we're always mutating at sinks, sources, or vertices that are a sink or source except for the arrow between $k$ and $k'$, the number of arrows between vertices $i$ and $j$ for $i,j \neq k,k' \in Q_0$ is fixed for all four quivers.
\end{proof}

Since mutation from pre-forks and keys are very similar, we can define a class of quivers that arise after mutating at their $k$ and $k'$ vertices.
These were first defined by Warkentin for pre-forks with $q_{kk'} = 0$ \cite[Definition 6.3]{warkentin_exchange_2014} and used quivers with polynomials as weights on the arrows for $|q_{kk'}| = 1$, but our definition shall hold in the case that $|q_{kk'}| = 1$.
Since all keys are abundant acyclic quivers and all pre-forks are forks when $|q_{kk'}| \geq 2$, we do not concern ourselves with the case $|q_{kk'}| \geq 2$.

\begin{Def} \label{def-wing}
    Let $W$ be a quiver.
    We say that $W$ is a \textit{wing} with point of return $k$ if it satisfies the following:
    \begin{itemize}
        \item $|w_{kk'}| < 2$ for some vertex $k'$.

        \item $W_{k'}^k = W_0 \setminus \{k,k'\}$.
    
        \item $W \setminus \{k'\}$ is a fork with point of return $k$.
        
        \item $W \setminus \{k\}$ is abundant acyclic.

        \item If one of
    
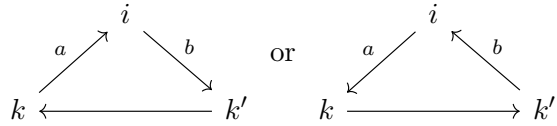
\[\begin{tikzcd}
        & i \\
        k & & k'
        \arrow[from=2-1,to=1-2, "a"]
        \arrow[from=1-2,to=2-3, "b"]
        \arrow[from=2-3,to=2-1]
    \end{tikzcd}
    \text{ or }
    \begin{tikzcd}
        & i \\
        k & & k'
        \arrow[from=1-2,to=2-1, "a"']
        \arrow[from=2-3,to=1-2, "b"']
        \arrow[from=2-1,to=2-3]
    \end{tikzcd}\]
    is a full subquiver of $W$ with $a, b > 0$, then $b \geq a + 2$.
    \end{itemize} 
\end{Def}

\begin{Cor} \label{cor-wing-pre-fork-key}
    Let $Q$ be a pre-fork with vertices $\{r,k,k'\}$ or a key with vertices $\{k,k'\}$.
    If $k$ and $k'$ are not sinks or sources, then $\mu_k(Q)$ and $\mu_{k'}(Q)$ are both wings with respective point of return $k$ and $k'$.
\end{Cor}

As before, we must explore what happens when we mutate a wing not at its point of return, where the proof for the $w_{kk'} = 0$ case has already been done by Warkentin.

\begin{Lem}{\textup{\cite[Lemmata 6.4 and 6.7]{warkentin_exchange_2014}}}\label{lem-wing-mutation-0}
    Let $W$ be a wing with point of return $k$ and $w_{kk'} =0$.
    If $r \neq k,k'$, then $\mu_r(W)$ is a fork with point of return $r$.
    If we mutate at $k'$, then $Q = \mu_{k'}(W)$ satisfies the following:
    \begin{itemize}
        \item $Q \setminus \{k\}$ is a fork with point of return $k'$.

        \item $Q_{k'}^k = \emptyset$.

        \item $Q \setminus \{k'\}$ is a fork with point of return $k$.

        \item $|Q_1| > |W_1|$.

        \item For $r \neq k,k'$, we have that $P = \mu_r(Q)$ is a pre-fork with vertices $\{r,k,k'\}$, where $|P_1| > |Q_1|$.
    \end{itemize}
\end{Lem}

\begin{Lem} \label{lem-wing-mutation-1}
    Let $W$ be a wing with point of return $k$ and $|w_{kk'}| = 1$.
    If $r \neq k,k'$, then $\mu_r(W)$ is a fork with point of return $r$.
    If we mutate at $k'$, then $Q = \mu_{k'}(W)$ satisfies the following:
    \begin{itemize}
        \item $Q \setminus \{k\}$ is a fork with point of return $k'$.

        \item $Q_{k'}^k = Q^-(k') \neq \emptyset$ if $q_{k'k} =  w_{kk'} = 1$ and $Q_{k'}^k = Q^+(k') \neq \emptyset$ if $q_{kk'} = w_{k'k} = 1$.

        \item $Q \setminus \{k'\}$ is abundant acyclic.
        Additionally, if $q_{k'k} = 1$, then $k$ is a source in $Q \setminus \{k'\}$, and if $q_{kk'} = 1$, then $k$ is a sink in $Q \setminus \{k'\}$.

        \item $|Q_1| > |W_1|$.

        \item For $r \in Q_{k'}^k$, we have that $F = \mu_r(Q)$ is a fork with point of return $r$.
        For $r \notin Q_{k'}^k$, we have that $F$ is a pre-fork with vertices $\{r,k,k'\}$.
        This forces $|F_1| > |Q_1|$ in either case.

        \item If one of
    
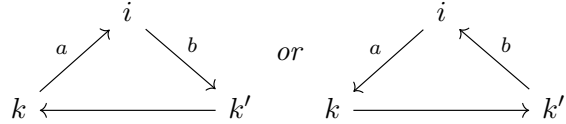
\[\begin{tikzcd}
        & i \\
        k & & k'
        \arrow[from=2-1,to=1-2, "a"]
        \arrow[from=1-2,to=2-3, "b"]
        \arrow[from=2-3,to=2-1]
    \end{tikzcd}
    \text{ or }
    \begin{tikzcd}
        & i \\
        k & & k'
        \arrow[from=1-2,to=2-1, "a"']
        \arrow[from=2-3,to=1-2, "b"']
        \arrow[from=2-1,to=2-3]
    \end{tikzcd}\]
    is a full subquiver of $Q$ with $a, b > 0$, then $b \geq a + 2$.
    \end{itemize}
\end{Lem}

\begin{proof}
    If $r \neq k,k'$, then $\mu_r(W)$ is a fork with point of return by the same argument used in the proof of Lemma \ref{lem-key-mutation-k}.
    We may then proceed with mutation at $k'$, with the first statement following from $W \setminus \{k\}$ being abundant acyclic and Lemma \ref{lem-fork-mutation}.

    Let $i \to k \to j$ be any path in $W$.
    Depending on the sign of $w_{kk'}$, we have one of the following two full subquivers of $W$:
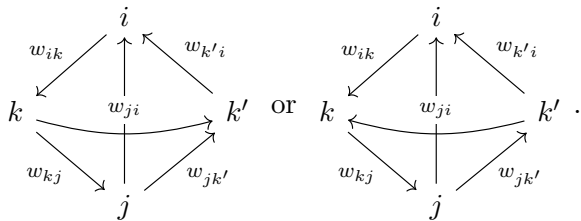
    \[
    \begin{tikzcd}
        & i \\
        k & & k' \\
        & j
        \arrow[from=1-2,to=2-1, "w_{ik}"']
        \arrow[from=2-3,to=1-2, "w_{k'i}"']
        \arrow[from=3-2,to=1-2, "w_{ji}" description]
        \arrow[from=3-2,to=2-3, "w_{jk'}"']
        \arrow[from=2-1,to=2-3, bend right=15]
        \arrow[from=2-1,to=3-2, "w_{kj}"']
    \end{tikzcd}
    \text{ or }
    \begin{tikzcd}
        & i \\
        k & & k' \\
        & j
        \arrow[from=1-2,to=2-1, "w_{ik}"']
        \arrow[from=2-3,to=1-2, "w_{k'i}"']
        \arrow[from=3-2,to=1-2, "w_{ji}" description]
        \arrow[from=3-2,to=2-3, "w_{jk'}"']
        \arrow[from=2-3,to=2-1, bend left=15]
        \arrow[from=2-1,to=3-2, "w_{kj}"']
    \end{tikzcd}.\]
    Mutation at $k'$ will produce the respective following full subquivers of $Q = \mu_{k'}(W)$:
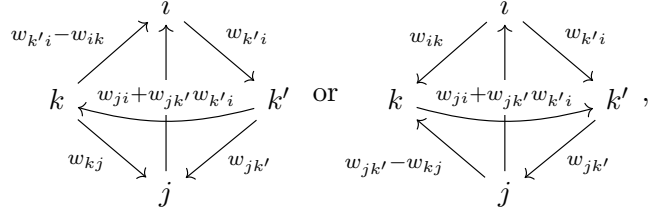
    \[
    \begin{tikzcd}
        & i \\
        k & & k' \\
        & j
        \arrow[from=2-1,to=1-2, "w_{k'i} - w_{ik}"]
        \arrow[from=1-2,to=2-3, "w_{k'i}"]
        \arrow[from=3-2,to=1-2, "w_{ji}+w_{jk'}w_{k'i}" description]
        \arrow[from=2-3,to=3-2, "w_{jk'}"]
        \arrow[from=2-3,to=2-1, bend left=15]
        \arrow[from=2-1,to=3-2, "w_{kj}"']
    \end{tikzcd}
    \text{ or}
    \begin{tikzcd}
        & i \\
        k & & k' \\
        & j
        \arrow[from=1-2,to=2-1, "w_{ik}"']
        \arrow[from=1-2,to=2-3, "w_{k'i}"]
        \arrow[from=3-2,to=1-2, "w_{ji}+w_{jk'}w_{k'i}" description]
        \arrow[from=2-3,to=3-2, "w_{jk'}"]
        \arrow[from=2-1,to=2-3, bend right=15]
        \arrow[from=3-2,to=2-1, "w_{jk'} - w_{kj}"]
    \end{tikzcd},\]
    as $w_{k'i} \geq w_{ik} + 2$ or $w_{jk'} \geq w_{kj} + 2$ respectively.
    This demonstrates that if $k \to v \to k'$ in $W$ for some vertex $v$ with $k \to k'$ or $k' \to k$, then $v \notin Q_{k'}^k$ or $v \in Q_{k'}^k$ respectively.
    Similarly, if $k' \to v \to k$ in $W$ for some vertex $v$ with $k \to k'$ or $k' \to k$, then $v \in Q_{k'}^k$ or $v \notin Q_{k'}^k$ respectively.
    Thus $Q_{k'}^k = W^+(k') \cap W^-(k)$ if $w_{kk'} = 1$ and $Q_{k'}^k = W^+(k) \cap W^-(k')$ if $w_{kk'} = -1$, proving most of the second statement.
    If $w_{kk'} = 1$, then $Q^-(k') = W^+(k') = W^-(k)$.
    If $w_{kk'} = -1$, then $Q^+(k') = W^-(k') = W^+(k)$.
    If $Q_{k'}^k$ did equal the empty set, then $k$ and $k'$ would both be a sink or a source in $W$, a contradiction.
    This completes the proof of the second statement.
    Additionally, since $w_{jk'}w_{k'i} \geq w_{ik}$ or $w_{jk'}w_{k'i} \geq w_{kj}$ respectively, there must be a positive number of arrows added to the subquivers after mutation at $k'$.
    As $i \to k \to j$ was arbitrary, this tells us that $|Q_1| > |W_1|$, proving the fourth statement.

    Continuing with the third statement, we already know that $Q \setminus \{k\}$ is a fork with point of return $k'$.
    Hence $Q \setminus \{k,k'\}$ is abundant acyclic.
    From the second statement, we know that $Q_{k'}^k = Q^-(k')$ if $w_{kk'} = 1$.
    This implies that for $i \neq k,k' \in Q_0$ we get the following two possible subquivers of $Q$:
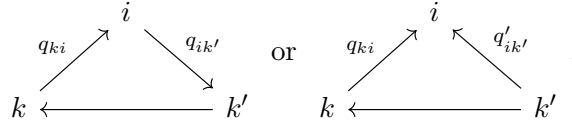
    \[\begin{tikzcd}    
        & i \\
        k & & k'
        \arrow[from=2-3,to=2-1]
        \arrow[from=2-1,to=1-2, "q_{ki}"]
        \arrow[from=1-2,to=2-3, "q_{ik'}"]
    \end{tikzcd}
    \text{ or }
    \begin{tikzcd}    
        & i \\
        k & & k'
        \arrow[from=2-3,to=2-1]
        \arrow[from=2-1,to=1-2, "q_{ki}"]
        \arrow[from=2-3,to=1-2, "q_{ik'}'"']
    \end{tikzcd}.
    \]
    We can then immediately see that $k$ must be a source in $Q \setminus \{k'\}$.
    A similar argument shows that $k$ must be a sink in $Q \setminus \{k'\}$ if $w_{k'k} = 1$.
    Attaching a source or a sink to an already acyclic quiver produces another acyclic quiver.
    Hence $Q \setminus \{k'\}$ must be acyclic.
    Its abundance follows from the property concerning 3-cycles involving $k$ and $k'$ of wings, proving the third statement.

    To prove the fifth statement, let $F = \mu_r(Q)$ for $r \neq k,k'$.
    Then $F \setminus \{k\}$ and $F \setminus \{k'\}$ are both forks with $r$ as the point of return.
    As $Q \setminus \{k\}$ was a fork with point of return $k'$ and $Q \setminus \{k'\}$ was abundant acyclic, mutation at $r$ forces $|(F \setminus \{k\})_1| > |(Q \setminus \{k\})_1|$ and $|(F \setminus \{k'\})_1| > |(Q \setminus \{k'\})_1|$.
    Thus we need only show that $|f_{kk'}| \geq |q_{kk'}|$ to prove that $|F_1| > |Q_1|$.
    
    If $r \in Q_{k'}^k$, then $|f_{kk'}| > |f_{kr}|, |f_{rk'}| \geq 2 > |q_{kk'}|$.
    This satisfies the first condition for a fork, and we need only show that $F^+(r)$ and $F^-(r)$ are acyclic to prove that $F$ is a fork with point of return $r$.
    However, both of those quivers will be full subquivers of either $F \setminus \{k,r\}$ or $F \setminus \{k',r\}$ as $r \in Q_{k'}^k$.
    As such, they are both acyclic and $F$ is a fork with point of return $r.$

    If $r \notin Q_{k'}^k$, then $f_{kk'} = q_{kk'}$.
    Hence, we have that $|F_1| > |Q_1|$ as long as $r \neq k,k'$.
    If $k \to v \to k'$ for some vertex $v$ in $F$, then either exactly one of the full subquivers induced by $\{r,k,v\}$ or $\{r,k',v\}$ is a 3-cycle or both are acyclic.
    The former contradicts the first condition of a fork for either $F \setminus \{k\}$ or $F \setminus \{k'\}$.
    The latter only occurs when $r$ is a source or sink in the full subquiver induced by $\{k,k',r,v\}$.
    Hence, we have arrived at one of the two following subquivers in $Q$:
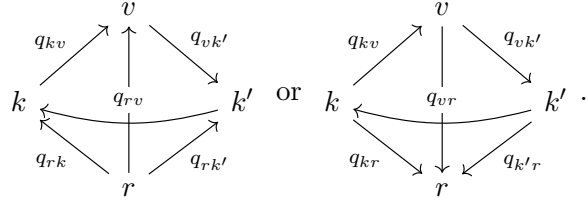
    \[
    \begin{tikzcd}
        & v \\
        k & & k' \\
        & r
        \arrow[from=1-2,to=2-3, "q_{vk'}"]
        \arrow[from=2-1,to=1-2, "q_{kv}"]
        \arrow[from=3-2,to=1-2, "q_{rv}" description]
        \arrow[from=3-2,to=2-3, "q_{rk'}"']
        \arrow[from=2-3,to=2-1, bend left=15]
        \arrow[from=3-2,to=2-1, "q_{rk}"]
    \end{tikzcd}
    \text{ or }
    \begin{tikzcd}
        & v \\
        k & & k' \\
        & r
        \arrow[from=1-2,to=2-3, "q_{vk'}"]
        \arrow[from=2-1,to=1-2, "q_{kv}"]
        \arrow[from=1-2,to=3-2, "q_{vr}" description]
        \arrow[from=2-3,to=3-2, "q_{k'r}"]
        \arrow[from=2-3,to=2-1, bend left=15]
        \arrow[from=2-1,to=3-2, "q_{kr}"']
    \end{tikzcd}.\]
    The second subquiver violates $Q \setminus \{k\}$ being a fork.
    The first subquiver violates $k$ being a source in $Q \setminus \{k'\}$.
    Note that the sign of $q_{kk'}$ is irrelevant here.
    Hence, there is no such $v$, and a similar argument shows that there is no $v$ such that $k' \to v \to k$.
    Thus $F_{k'}^k = \emptyset$, and $F$ is a pre-fork with vertices $\{r,k,k'\}$, proving the fifth statement.

    Finally, suppose that 
    
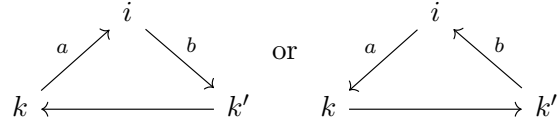
\[\begin{tikzcd}
        & i \\
        k & & k'
        \arrow[from=2-1,to=1-2, "a"]
        \arrow[from=1-2,to=2-3, "b"]
        \arrow[from=2-3,to=2-1]
    \end{tikzcd}
    \text{ or }
    \begin{tikzcd}
        & i \\
        k & & k'
        \arrow[from=1-2,to=2-1, "a"']
        \arrow[from=2-3,to=1-2, "b"']
        \arrow[from=2-1,to=2-3]
    \end{tikzcd}\]
    is a full subquiver of $Q$ with $a, b > 0$.
    Since $Q_{k'}^k = W^+(k') \cap W^-(k)$ if $w_{kk'} = 1$ and $Q_{k'}^k = W^+(k) \cap W^-(k')$ if $w_{k'k} = 1$, we get that mutation at $k'$ must produce
    
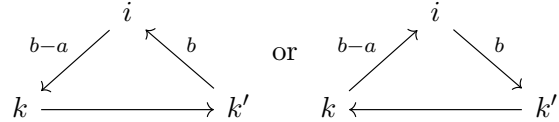
\[
    \begin{tikzcd}
        & i \\
        k & & k'
        \arrow[from=1-2,to=2-1, "b-a"']
        \arrow[from=2-3,to=1-2, "b"']
        \arrow[from=2-1,to=2-3]
    \end{tikzcd}
    \text{ or }
    \begin{tikzcd}
        & i \\
        k & & k'
        \arrow[from=2-1,to=1-2, "b-a"]
        \arrow[from=1-2,to=2-3, "b"]
        \arrow[from=2-3,to=2-1]
    \end{tikzcd}\]
    in $W$.
    Thus $b \geq a + 2$, proving the sixth and final statement.
\end{proof}

This leads us to our final type of quivers and two immediate corollaries about them, coming from Lemmata \ref{lem-wing-mutation-0} and \ref{lem-wing-mutation-1}.
Again, these quivers were first defined by Warkentin when dealing with the case of $q_{kk'} = 0$ \cite[Definition 6.3]{warkentin_exchange_2014}, but we provide a more general definition that will work in the case $|q_{kk'}| = 1$.

\begin{Def} \label{def-tips}
    Let $T$ be a quiver.
    We say that $T$ is a \textit{tip} with point of return $k'$ if it satisfies the following:
    \begin{itemize}
        \item $|t_{kk'}| < 2$ for some vertex $k$.
    
        \item $T \setminus \{k\}$ is a fork with point of return $k'$.

        \item $T_{k'}^k = \emptyset$ if $t_{k'k} = 0$, $T_{k'}^k = T^-(k')$ if $t_{k'k} = 1$, and $T_{k'}^k = T^+(k')$ if $t_{kk'} = 1$.

        \item $T \setminus \{k'\}$ is a fork with point of return $k$ if $t_{kk'} = 0$ and abundant acyclic otherwise.
        Additionally, if $t_{k'k} = 1$, then $k$ is a source in $T \setminus \{k'\}$, and if $t_{kk'} = 1$, then $k$ is a sink in $T\setminus \{k'\}$.

        \item If one of 
    
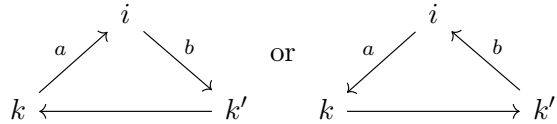
\[\begin{tikzcd}
        & i \\
        k & & k'
        \arrow[from=2-1,to=1-2, "a"]
        \arrow[from=1-2,to=2-3, "b"]
        \arrow[from=2-3,to=2-1]
    \end{tikzcd}
    \text{ or }
    \begin{tikzcd}
        & i \\
        k & & k'
        \arrow[from=1-2,to=2-1, "a"']
        \arrow[from=2-3,to=1-2, "b"']
        \arrow[from=2-1,to=2-3]
    \end{tikzcd}\]
    is a full subquiver of $T$ with $a, b > 0$, then $b \geq a + 2$.
    \end{itemize}
\end{Def}

Note that our point of return is only unique if $t_{kk'} \neq 0$.

\begin{Cor} \label{cor-tip-wing}
    Let $W$ be a wing with point of return $k$.
    Then $T = \mu_{k'}(W)$ is a tip with point of return $k'$, and $|T_1| > |W_1|$.
\end{Cor}

\begin{Cor} \label{cor-tip-mutation}
    Let $T$ be a tip with point of return $k'$ and $r \neq k,k' $ a vertex.
    Then $\mu_r(T)$ is a pre-fork with vertices $\{r,k,k'\}$ if $t_{kk'} = 0$ or $r \in T_{k'}^k$, and $\mu_r(T)$ is a fork with point of return $r$ if $t_{kk'} \neq 0$ and $r \notin T_{k'}^k$.
    In either case, the quiver $P = \mu_r(T)$ has more arrows than $T$, i.e., $|P_1| > |T_1|$.
\end{Cor}

Finally, we may describe all the quivers mutation-equivalent to a key that are not arrived at by passing through a pre-fork.
Lemma \ref{lem-key-subgraph} is the culmination of our previous definitions, corollaries, and lemmata, as we show that a key's labelled mutation class contains a connected subgraph devoid of pre-forks or forks.
Furthermore, this subgraph contains a finite number of quivers.
Drawing on our experience with abundant acyclic quivers, we can reasonably suspect that there will be a portion of any quiver's labelled mutation class that is analogous to the forkless part for pre-forks instead.
Discussion of the pre-forkless part is reserved for Section \ref{sec-fpfp}.

\begin{Lem} \label{lem-key-subgraph}
    Let $Q$ be a key with vertices $\{k,k'\}$ on $n$ vertices.
    Then we see the following for the different types of keys:
    
    If $q_{kk'} = 0$, then we have the following subgraph of the labelled mutation graph of $Q$:
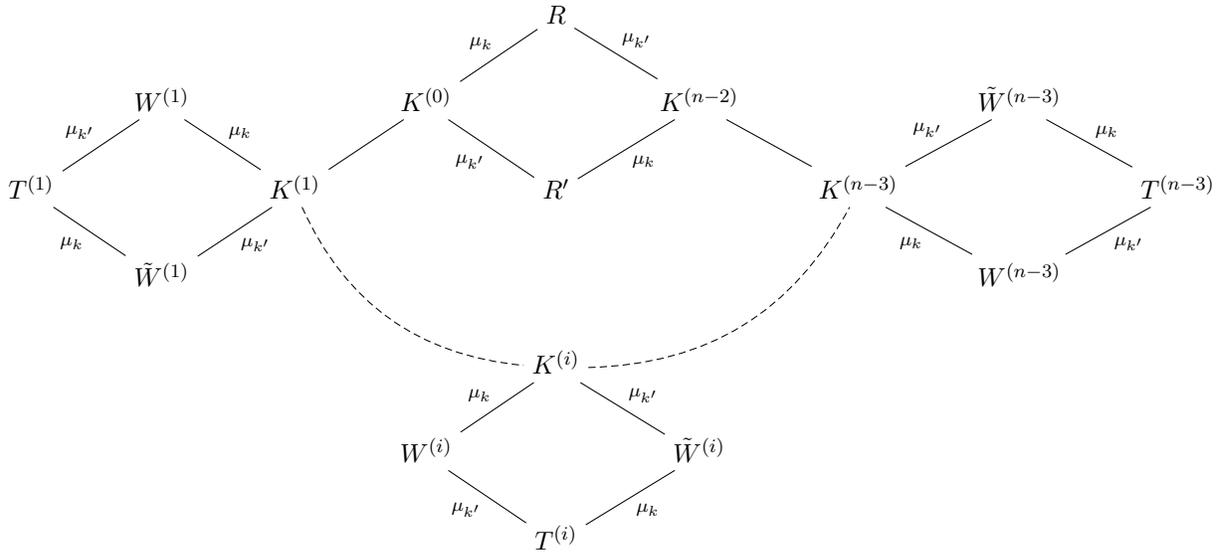
    \[
        \adjustbox{scale=0.9, center}{%
        \begin{tikzcd}[cramped, ampersand replacement=\&]
            \& \& \& \& R \\
            \& W^{(1)} \& \& K^{(0)} \& \& K^{(n-2)} \& \& \Tilde{W}^{(n-3)} \\
            T^{(1)} \& \& K^{(1)} \& \& R' \& \& K^{(n-3)} \& \& T^{(n-3)}\\
            \& \Tilde{W}^{(1)} \& \& \& \& \& \& W^{(n-3)}\\
            \& \& \& \& K^{(i)}\\
            \& \& \& W^{(i)} \& \& \Tilde{W}^{(i)}\\
            \& \& \& \& T^{(i)}
            \arrow[from=1-5, to=2-4, no head, "\mu_{k}"']
            \arrow[from=1-5, to=2-6, no head, "\mu_{k'}"]
            \arrow[from=3-5, to=2-4, no head, "\mu_{k'}"]
            \arrow[from=3-5, to=2-6, no head, "\mu_{k}"']
            \arrow[from=3-3, to=2-4, no head]
            \arrow[from=3-7, to=2-6, no head]
            \arrow[from=2-2, to=3-1, no head, "\mu_{k'}"']
            \arrow[from=2-2, to=3-3, no head, "\mu_{k}"]
            \arrow[from=2-8, to=3-7, no head, "\mu_{k'}"']
            \arrow[from=2-8, to=3-9, no head, "\mu_{k}"]
            \arrow[from=4-2, to=3-1, no head, "\mu_{k}"]
            \arrow[from=4-2, to=3-3, no head, "\mu_{k'}"']
            \arrow[from=4-8, to=3-7, no head, "\mu_{k}"]
            \arrow[from=4-8, to=3-9, no head, "\mu_{k'}"']
            \arrow[from=3-3, to=5-5, no head, dashed, bend right]
            \arrow[from=5-5, to=3-7, no head, dashed, bend right]
            \arrow[from=5-5, to=6-4, no head, "\mu_{k}"']
            \arrow[from=5-5, to=6-6, no head, "\mu_{k'}"]
            \arrow[from=7-5, to=6-4, no head, "\mu_{k'}"]
            \arrow[from=7-5, to=6-6, no head, "\mu_{k}"']
        \end{tikzcd}
        }\]
        where 
        \begin{itemize}
            \item $K^{(0)}$ is a key with $k$ a source;
            
            \item $K^{(n-2)}$ is a key with $k$ a sink;
            
            \item $K^{(i)}$ is a key with $k$ not a source obtained by mutating repeatedly $K^{(0)}$ at $i$ sinks for $i > 0$;

            \item $Q$ is one of the $K^{(i)}$;

            \item $W^{(i)}$ and $\Tilde{W}^{(i)}$ are both wings for $i > 0$;

            \item $T^{(i)}$ is a tip for $i > 0$;
            
            \item $R$ and $R'$ are both acyclic quivers;

            \item mutation at any vertex $r$ that is not labelled and not a sink or source produces a pre-fork with vertices $\{r,k,k'\}$ or a fork with point of return $r$.
        \end{itemize}

    If $|q_{kk'}| = 1$ then we have the following subgraph of the labelled mutation graph of $Q$:
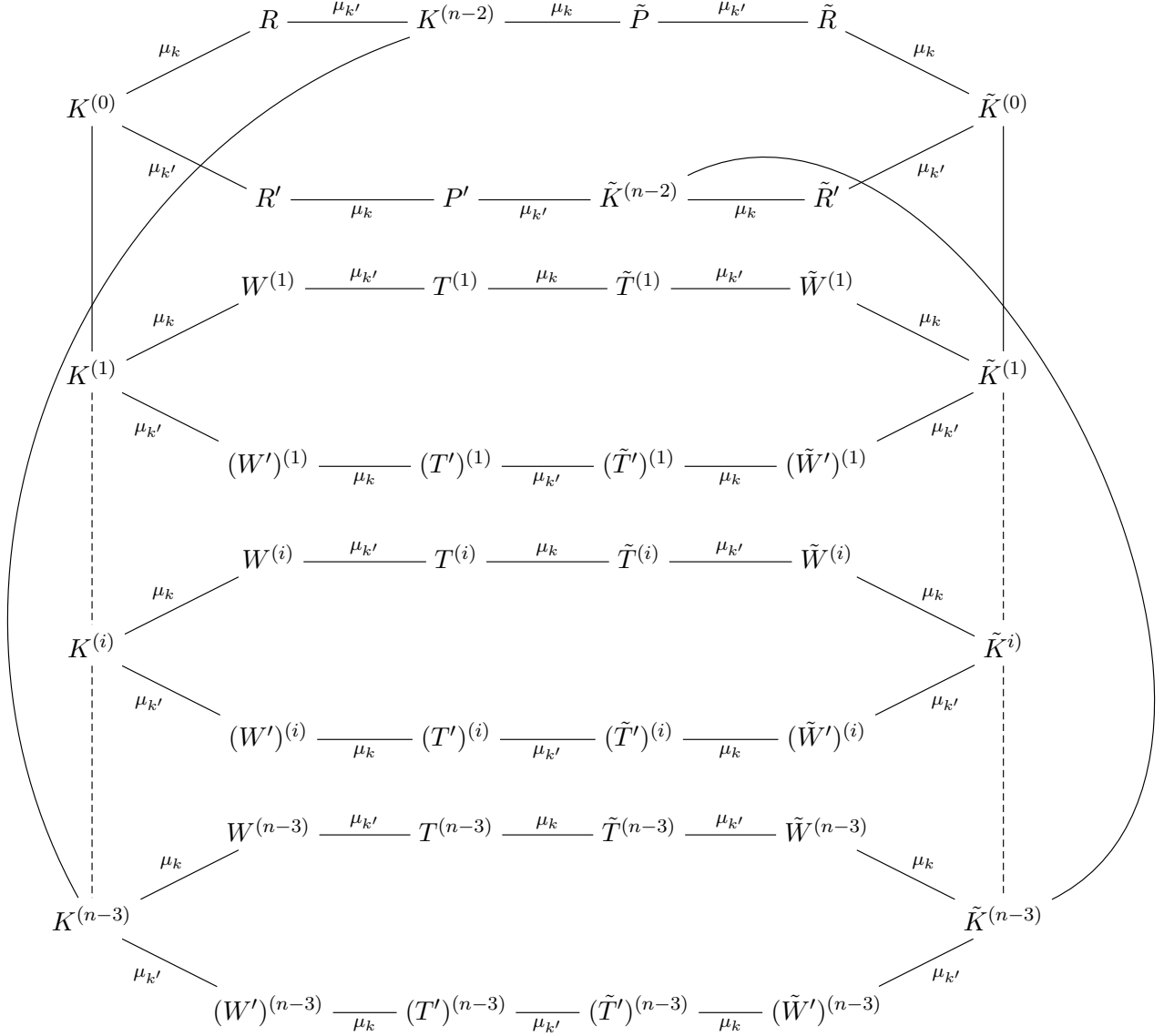
    \[
    \begin{tikzcd}[cramped]
            &  R & K^{(n-2)} & \Tilde{P} & \Tilde{R} &  \\
            K^{(0)} & & & & &   \Tilde{K}^{(0)}  \\
            &  R' & P' & \Tilde{K}^{(n-2)} & \Tilde{R}' &   \\
            &  W^{(1)} & T^{(1)} & \Tilde{T}^{(1)} & \Tilde{W}^{(1)} &  \\
            K^{(1)} & & & & &   \Tilde{K}^{(1)} \\
            &  (W')^{(1)} & (T')^{(1)} & (\Tilde{T}')^{(1)} & (\Tilde{W}')^{(1)} & \\
            &  W^{(i)} & T^{(i)} & \Tilde{T}^{(i)} & \Tilde{W}^{(i)} &  \\
            K^{(i)} & & & & &   \Tilde{K}^{i)} \\
            &  (W')^{(i)} & (T')^{(i)} & (\Tilde{T}')^{(i)} & (\Tilde{W}')^{(i)} & \\
            &  W^{(n-3)} & T^{(n-3)} & \Tilde{T}^{(n-3)} & \Tilde{W}^{(n-3)} &  \\
            K^{(n-3)} & & & & &   \Tilde{K}^{(n-3)} \\
            &  (W')^{(n-3)} & (T')^{(n-3)} & (\Tilde{T}')^{(n-3)} & (\Tilde{W}')^{(n-3)}
            \arrow[from=2-1,to=1-2, no head, "\mu_{k}"]
            \arrow[from=2-1,to=3-2, no head, "\mu_{k'}"']
            \arrow[from=1-2,to=1-3, no head, "\mu_{k'}"]
            \arrow[from=3-2,to=3-3, no head, "\mu_{k}"']
            \arrow[from=1-3,to=1-4, no head, "\mu_{k}"]
            \arrow[from=3-3,to=3-4, no head, "\mu_{k'}"']
            \arrow[from=1-4,to=1-5, no head, "\mu_{k'}"]
            \arrow[from=3-4,to=3-5, no head, "\mu_{k}"']
            \arrow[from=1-5,to=2-6, no head, "\mu_{k}"]
            \arrow[from=3-5,to=2-6, no head, "\mu_{k'}"']
            \arrow[from=2-1,to=5-1, no head]
            \arrow[from=2-6,to=5-6, no head]
            \arrow[from=5-1,to=8-1, no head, dashed]
            \arrow[from=5-6,to=8-6, no head, dashed]
            \arrow[from=5-1,to=4-2, no head, "\mu_{k}"]
            \arrow[from=5-1,to=6-2, no head, "\mu_{k'}"']
            \arrow[from=4-2,to=4-3, no head, "\mu_{k'}"]
            \arrow[from=6-2,to=6-3, no head, "\mu_{k}"']
            \arrow[from=4-3,to=4-4, no head, "\mu_{k}"]
            \arrow[from=6-3,to=6-4, no head, "\mu_{k'}"']
            \arrow[from=4-4,to=4-5, no head, "\mu_{k'}"]
            \arrow[from=6-4,to=6-5, no head, "\mu_{k}"']
            \arrow[from=4-5,to=5-6, no head, "\mu_{k}"]
            \arrow[from=6-5,to=5-6, no head, "\mu_{k'}"']
            \arrow[from=8-1,to=7-2, no head, "\mu_{k}"]
            \arrow[from=8-1,to=9-2, no head, "\mu_{k'}"']
            \arrow[from=7-2,to=7-3, no head, "\mu_{k'}"]
            \arrow[from=9-2,to=9-3, no head, "\mu_{k}"']
            \arrow[from=7-3,to=7-4, no head, "\mu_{k}"]
            \arrow[from=9-3,to=9-4, no head, "\mu_{k'}"']
            \arrow[from=7-4,to=7-5, no head, "\mu_{k'}"]
            \arrow[from=9-4,to=9-5, no head, "\mu_{k}"']
            \arrow[from=7-5,to=8-6, no head, "\mu_{k}"]
            \arrow[from=9-5,to=8-6, no head, "\mu_{k'}"']
            \arrow[from=8-1,to=11-1, no head, dashed]
            \arrow[from=8-6,to=11-6, no head, dashed]
            \arrow[from=11-1,to=10-2, no head, "\mu_{k}"]
            \arrow[from=11-1,to=12-2, no head, "\mu_{k'}"']
            \arrow[from=10-2,to=10-3, no head, "\mu_{k'}"]
            \arrow[from=12-2,to=12-3, no head, "\mu_{k}"']
            \arrow[from=10-3,to=10-4, no head, "\mu_{k}"]
            \arrow[from=12-3,to=12-4, no head, "\mu_{k'}"']
            \arrow[from=10-4,to=10-5, no head, "\mu_{k'}"]
            \arrow[from=12-4,to=12-5, no head, "\mu_{k}"']
            \arrow[from=10-5,to=11-6, no head, "\mu_{k}"]
            \arrow[from=12-5,to=11-6, no head, "\mu_{k'}"']
            \arrow[from=11-1,to=1-3, overlay, no head, bend left=50]
            \arrow[from=11-6,to=3-4, overlay, no head, bend right=90]
    \end{tikzcd}
    \]
        where 
        \begin{itemize}
            \item $K^{(0)}$ is a key with $k$ a source;

            \item $\Tilde{K}^{(0)}$ is a key with $k'$ as a source;

            \item $K^{(n-2)}$ is a key with $k'$ a sink;

            \item $\Tilde{K}^{(n-2)}$ is a key with $k$ as a sink;

            \item $K^{(i)}$ is a key with $k$ not a source obtained by mutating repeatedly $K^{(0)}$ at $i$ sinks for $i > 0$;
            
            \item $Q$ is one of the $K^{(i)}$ or $\Tilde{K}^{(i)}$

            \item $\Tilde{K}^{(i)}$ is a key with $k'$ not a source obtained by mutating repeatedly $\Tilde{K}^{(0)}$ at $i$ sinks for $i > 0$;

            \item $W^{(i)}$, $\Tilde{W}^{(i)}$, $(W')^{(i)}$, and $(\Tilde{W}')^{(i)}$ are all wings for $i > 0$;

            \item $T^{(i)}$, $\Tilde{T}^{(i)}$, $(T')^{(i)}$, and $(\Tilde{T}')^{(i)}$ are all tips for $i > 0$;

            \item mutation at any vertex $r$ that is not labelled and not a sink or source produces a pre-fork with vertices $\{r,k,k'\}$ or a fork with point of return $r$.
        \end{itemize}
\end{Lem}

\begin{proof}
    We suppose first that $q_{kk'} = 0$.
    From Corollary \ref{cor-key-source-sink} and Lemma \ref{lem-key-mutation}, we may mutate $Q$ along sources to a quiver $K^{(0)}$ with $k$ a source and along sinks to $K^{(n-2)}$ with $k$ a sink.
    Let $X = K^{(0)}$ and $Y = K^{(n-2)}$.
    Then $y_{ij}$ must equal $x_{ij}$ for all $i,j \neq k,k'$, as we mutate at every vertex other than $k$ or $k'$ exactly once to reach $K^{(n-2)}$.
    This also tells us that $y_{ik} = -x_{ik}$, $y_{ik'} = -x_{ik'}$, and $y_{kk'} = x_{kk'}$ for all $i \neq k,k'$.
    Thus, Lemma \ref{lem-key-k-sink-source-mutation} tells us that $K^{(n-2)} = \mu_{[k,k']}(K^{(0)}) = \mu_{[k',k]}(K^{(0)})$, completing the cycle.
    
    Additionally, every quiver $K^{(i)}$ along the way must be a key with vertices $\{k,k'\}$ for $0 < i < n-2$.
    Thus mutating at $k$ or $k'$ produces a wing $W^{(i)}$ and along $[k,k']$ or $[k',k]$ produces a tip $T^{(i)}$ by Corollaries \ref{cor-wing-pre-fork-key} and \ref{cor-tip-wing}, meaning mutation at $r \neq k,k'$ will produce either a pre-fork with vertices $\{r,k,k'\}$ or a fork with point of return $r$.
    Lemma \ref{lem-key-k-sink-source-mutation} tells us that $R$ and $R'$ are acyclic quivers such that $\mu_r(R)$ and $\mu_r(R')$ are forks with point of return $r$.
    Thus the above connected subgraph appears in the labelled mutation graph of $Q$, and mutating at a vertex $r$ that is not on a labelled  edge and not a sink or source will produce a pre-fork with vertices $\{r,k,k'\}$ or a fork with point of return $r$.

    If $|q_{kk'}| = 1$, we may mutate $Q$ in the same manner that we did before along sources to a quiver $K^{(0)}$ (or an isomorphic copy with $k$ and $k'$ swapped) with $k$ a source and along sinks to $K^{(n-2)}$ with $k'$ a sink.
    Every quiver $K^{(i)}$ along the way with $0 < i < n-2$ must be a key with vertices $\{k,k'\}$, where neither $k$ or $k'$ is a sink or source.
    Thus mutating $K^{(i)}$ at $k$ or $k'$ produces a wing $W^{(i)}$ or $(W')^{(i)}$ and along $[k,k']$ or $[k',k]$ produces a tip $T^{(i)}$ or $(T')^{(i)}$ by Corollaries \ref{cor-wing-pre-fork-key} and \ref{cor-tip-wing}, meaning mutation at $r \neq k,k'$ will produce either a pre-fork with vertices $\{r,k,k'\}$ or a fork with point of return $r$.
    Furthermore, mutating each $K^{(i)}$ along $[k,k',k,k',k]$ or $[k',k,k',k,k']$ produces an isomorphic copy $\Tilde{K}^{(i)}$, making them keys with vertices $\{k,k'\}$.
    This then shows that $\Tilde{K}^{(0)}$ is a key with source $k'$ and $\Tilde{K}^{(n-2)}$ is a key with sink $k$.
    The $R$, $R'$, $P$, and $P'$ quivers are quivers such that mutation at $r \neq k,k'$ produces a fork with point of return $r$ by Lemma \ref{lem-key-k-sink-source-mutation}.
    We need only show that the cycle involving the keys $K^{(0)}$ through $K^{(n-2)}$ exists.

    To do so, let $X = K^{(0)}$ and $Y = K^{(n-2)}$.
    Then $y_{ij}$ must equal $x_{ij}$ for all $i,j \neq k,k'$, as we mutate at every vertex other than $k$ or $k'$ exactly once to reach $K^{(n-2)}$.
    This also tells us that $y_{ik} = -x_{ik}$, $y_{ik'} = -x_{ik'}$, and $y_{kk'} = x_{kk'}$ for all $i \neq k,k'$.
    Thus, Lemma \ref{lem-key-k-sink-source-mutation} tells us that $K^{(n-2)} = \mu_{[k,k']}(K^{(0)})$, completing the cycle.
    A similar argument demonstrates the cycle formed from the $\Tilde{K}^{(i)}$.
\end{proof}

\begin{Rmk} \label{rmk-keys}
    The shape produced in the labelled mutation graph when $q_{kk'} = 0$ inspired the name keys, as it is reminiscent of keys on a keyring.
\end{Rmk}

We now need to show a version of Warkentin's Tree Lemma, allowing us to ignore the quivers that are obtained after mutating through a pre-fork.
He proved our Lemma \ref{lem-pre-fork-tree} for the exchange graph instead of the labelled mutation graph \cite[Lemma 6.7 \& Lemma 13.20]{warkentin_exchange_2014}.
We prove it for the labelled mutation graph in order to better control the mutation sequences between quivers in the labelled mutation class.

\begin{Lem} \label{lem-pre-fork-tree}
    Let $Q$ be a pre-fork with vertices $\{r,k,k'\}$.
    \begin{enumerate}[label=\alph*)]
        \item \textup{\cite[Lemma 6.7]{warkentin_exchange_2014}} If $q_{kk'} = 0$, then we get the following subgraph in the labelled mutation graph:
    
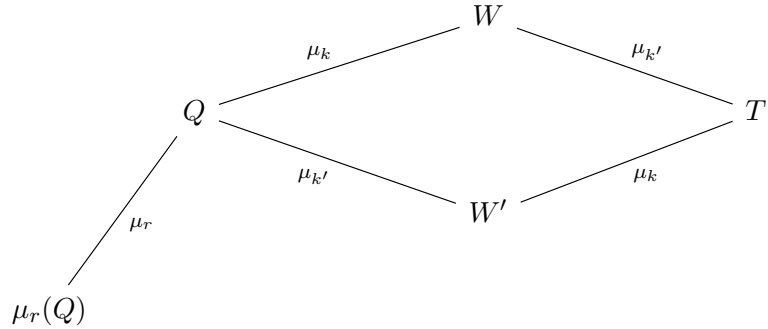
\[\begin{tikzcd}
    & & & & W & & & \\
    & Q & & & & & & T \\
    & & & & W' & & &   \\
    \mu_r(Q) & &
    \arrow[from=2-2, to=1-5, no head, "\mu_{k}"]
    \arrow[from=2-2, to=3-5, no head, "\mu_{k'}"']
    \arrow[from=2-2, to=4-1, no head, "\mu_r"]
    \arrow[from=1-5, to=2-8, no head, "\mu_{k'}"]
    \arrow[from=3-5, to=2-8, no head, "\mu_{k}"']
    \end{tikzcd}\]
    where
    \begin{itemize}
        \item $W$ is a wing with point of return $k$;

        \item $W'$ a wing with point of return $k'$;

        \item $T$ a tip with point of return $k$;

        \item mutation at a non-labelled vertex $m$ in a quiver other than $\mu_r(Q)$ produces either a fork with point of return $m$ or a pre-fork with vertices $\{m,k,k'\}$.
    \end{itemize}

        \item If $|q_{kk'}| = 1$, then we get the following subgraph in the labelled mutation graph:
    
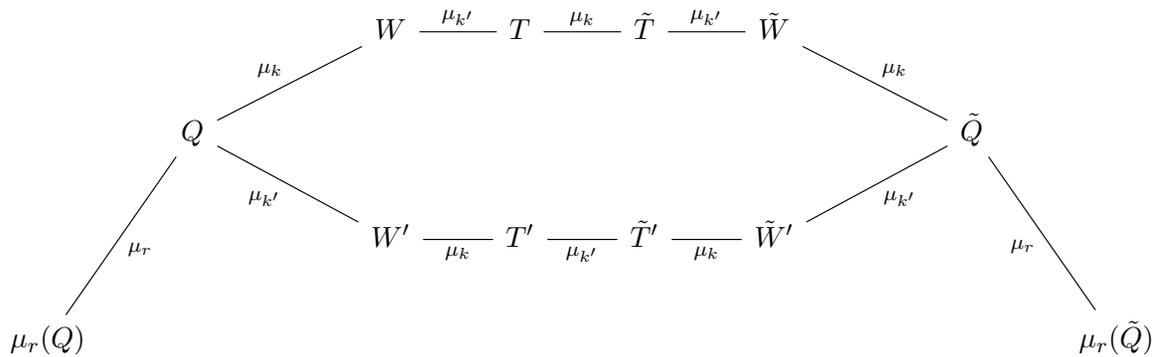
\[\begin{tikzcd}
    & & & W & T & \Tilde{T} & \Tilde{W} & & \\
    & Q & & & & & & & \Tilde{Q} & \\
    & & & W' & T' & \Tilde{T}' & \Tilde{W}' & &  \\
    \mu_r(Q) & & & & & & & & & \mu_r(\Tilde{Q})
    \arrow[from=2-2, to=1-4, no head, "\mu_{k}"]
    \arrow[from=2-2, to=3-4, no head, "\mu_{k'}"']
    \arrow[from=2-2, to=4-1, no head, "\mu_r"]
    \arrow[from=2-9, to=4-10, no head, "\mu_r"']
    \arrow[from=1-4, to=1-5, no head, "\mu_{k'}"]
    \arrow[from=1-5, to=1-6, no head, "\mu_{k}"]
    \arrow[from=1-6, to=1-7, no head, "\mu_{k'}"]
    \arrow[from=1-7, to=2-9, no head, "\mu_{k}"]
    \arrow[from=3-4, to=3-5, no head, "\mu_{k}"']
    \arrow[from=3-5, to=3-6, no head, "\mu_{k'}"']
    \arrow[from=3-6, to=3-7, no head, "\mu_{k}"']
    \arrow[from=3-7, to=2-9, no head, "\mu_{k'}"']
    \end{tikzcd}\]
    where
    \begin{itemize}
        \item $W, W', \Tilde{W}, $ and $\Tilde{W}'$ are wings;

        \item $T, T', \Tilde{T}, $ and $\Tilde{T}'$ are tips;

        \item and mutation at a non-labelled vertex $m$ in a quiver other than $\mu_r(Q)$ or $\mu_r(\Tilde{Q})$ produces either a fork with point of return $m$ or a pre-fork with vertices $\{m,k,k'\}$.
    \end{itemize}
    Additionally, we have that $W \cong \Tilde{W}'$, $T \cong \Tilde{T}'$, $\Tilde{T} \cong T'$, $\Tilde{W} \cong W'$, $Q \cong \Tilde{Q}$, and $\mu_r(Q) \cong \mu_r(\Tilde{Q})$ all under the isomorphism that fixes every vertex but $k$ and $k'$.
    Furthermore, performing the mutation $[k,k',k,k',k]$ or $[k',k,k',k,k']$ on $\mu_r(Q)$ produces $\mu_r(\Tilde{Q})$ and vice versa.
    \end{enumerate}
\end{Lem}

\begin{proof}
    For the case that $q_{kk'} = 0$, mutating along $[k,k',k,k']$ or $[k',k,k',k]$ returns us to $Q$.
    Thus Corollaries \ref{cor-wing-pre-fork-key} and \ref{cor-tip-wing} and Lemma \ref{lem-wing-mutation-0} show the desired result.

    If $|q_{kk'}| = 1$, then mutating $Q$ along $[k,k',k,k',k]$ or $[k',k,k',k,k']$ produces the same quiver $\Tilde{Q}$, which is an isomorphic copy of $Q$ obtained by swapping $k$ and $k'$.
    As such, all of the $W$ quivers are wings, and all of the $T$ quivers are tips by Corollaries \ref{cor-wing-pre-fork-key} and \ref{cor-tip-wing}.
    Since $Q$ is isomorphic to $\Tilde{Q}$ by swapping $k$ and $k'$, the same can be sad about $\mu_r(Q)$ and $\mu_r(\Tilde{Q})$.
    Then $Q_{k'}^k = \emptyset$ implies that there is still a single arrow between $k$ and $k'$ in $\mu_r(Q)$.
    Hence, mutating along $[k,k',k,k',k]$ or $[k',k,k',k,k']$ produces the same quiver $\mu_r(\Tilde{Q})$.
    All of the isomorphisms follow.
\end{proof}

\begin{Cor} \label{cor-num-arrows}
    Let $Q$ be a pre-fork with vertices $\{r,k,k'\}$.
    Then all of the other quivers, excepting $\mu_r(Q)$ and $\mu_r(\Tilde{Q})$ in Lemma \ref{lem-pre-fork-tree} have more arrows than $Q$.
\end{Cor}

This brings us to another slight generalization from the exchange graph to the labelled mutation graph \cite[Lemma 6.7 \& Lemma 13.21]{warkentin_exchange_2014}.

\begin{Cor} \label{cor-pre-fork-connected-subgraph}
    Let $Q$ be a pre-fork with vertices $\{r,k,k'\}$.
    Then removing the edges labelled $\mu_r$ that are touching $Q$ or $\Tilde{Q}$ as in Lemma \ref{lem-pre-fork-tree} produces two connected subgraphs in the labelled mutation graph of $Q$, one containing $\mu_r(Q)$ and the other containing $Q$.
    In particular, if $P = \mu_\bw(Q)$ for any mutation sequence that stays in the connected subgraph containing $Q$, then $P$ is a tip, wing, pre-fork, or fork such that $|P_1| \geq |Q_1|$, where equality implies that $P = Q$ or $P = \Tilde{Q}$.
\end{Cor}

\begin{proof}
    Assume that $|q_{kk'}| = 1$.
    We will induct on the length, $m$, of a mutation sequence $\bw$ such that the number of times, $t$, the sequence $[k,k',k,k',k]$ or $[k',k,k',k,k']$ appears in $\bw$ is zero.
    We will show two things:
    \begin{enumerate}
        \item $|Q_1| < |Q^{(1)}_1| < \dots < |Q^{(s)}_1| < |P_1|$, where the $Q^{(i)}$ are the pre-forks or forks in the order that they appear along the mutation sequence $\bw$;

        \item $P$ is a tip, wing, pre-fork, or fork, where the point of return is either $k, k'$, or the vertex that we last mutated at.
    \end{enumerate}
    
    For the base case $m = 0$, it naturally follows that $P = Q$, which proves our two desired results.
    Suppose now that $\bw$ is a mutation sequence with $t = 0$ and that $v$ is a vertex different from the one last mutated at.
    Let $P' = \mu_{\bw[v]}(Q)$.
    First note that we cannot have ever mutated at the point of return of a pre-fork or fork along $\bw$, as that would contradict $(1)$ for a mutation sequence of length less than $m$ by Lemmata \ref{lem-pre-fork-mutation} and \ref{lem-wing-mutation-1} and by Corollary \ref{cor-tip-mutation}.
    Additionally, if $P$ is a wing and $v$ is its point of return, then our previous statement implies that the last five mutations $\bw[v]$ would be either $[k,k',k,k',k]$ or $[k',k,k',k,k']$, a contradiction of $t = 0$.
    This also proves that $P$ is not a pre-fork or fork with point of return $v$.

    As such, if $P$ is a tip, wing, pre-fork, or fork, then $P'$ is a tip, wing, pre-fork, or fork with point of return $k, k'$, or $v$ by Lemmata \ref{lem-pre-fork-mutation}, \ref{lem-wing-mutation-1}, and \ref{lem-pre-fork-tree}, proving (2).
    While we may not have that $|P'_1| > |P_1|$ (consider the case of mutating a tip with point of return $k'$ at $k$ or vice versa), Corollary \ref{cor-num-arrows} proves that $|P'_1| > |Q^{(s)}_1|$ if $P'$ is a tip or wing.
    The previous Lemmata prove the same if $P$ is a pre-fork or fork.
    Hence, we have inductively shown (1) and (2) for all mutation sequences $\bw$ that do not contain either $[k,k',k,k',k]$ or $[k',k,k',k,k']$.
    This proves that $P$ is a tip, wing, pre-fork, or fork such that $|P_1| \geq |Q_1|$, where equality implies that $P = Q$, for all mutation sequences $\bw$ with $t = 0$.
    
    Now, assume that $P = \mu_{\bw}(Q)$ for some mutation sequence $\bw$ with $t > 0$ such that $\bw$ does not go through the edge labelled by $\mu_r$ connecting $Q$ and $\mu_r(Q)$ or $\Tilde{Q}$ and $\mu_r(\Tilde{Q})$.
    Form the mutation sequence $\bw'$ by removing the first $[k,k',k,k',k]$ or $[k',k,k',k,k']$ to appear in $\bw$ and swapping every instance of $k$ with $k'$ and of $k'$ with $k$ appearing after the removal.
    For example, the mutation sequence $\bw = [k,i,k,k',k,k',k,j,k,k']$ would become $\bw' = [k,i,j,k',k]$.
    Let $P' = \mu_{\bw'}(Q)$.
    Then $P'$ must be isomorphic to $P$, where the isomorphism is found by swapping $k$ and $k'$ and fixing the remaining vertices.
    We may continue repeating this process until we arrive at a $\Tilde{P} = \mu_{\Tilde{\bw}}(Q)$, where $\Tilde{\bw}$ has no $[k,k',k,k',k]$ or $[k',k,k',k,k']$ occurring.
    This forces $\Tilde{P}$ to be equal to $P$ or isomorphic to $P$ by swapping $k$ and $k'$.
    By our previous induction, we know that $\Tilde{P}$ is a tip, wing, pre-fork, or fork such that $|\Tilde{P}_1| \geq |Q_1|$, where equality only occurs if $\Tilde{P} = Q$, meaning $\Tilde{\bw} = []$, the empty mutation.
    Therefore, the quiver $P = \mu_\bw(Q)$ is a tip, wing, pre-fork, or fork such that $|P_1| \geq |Q_1|$, where equality implies that $P = Q$ or $\Tilde{Q}$, for all mutation sequences $\bw$ that stay in the connected subgraph containing $Q$.
    This implies that $\mu_r(Q)$ is not contained in the connected subgraph containing $Q$ and that the removal of the edges $\mu_r$ splits the labelled mutation graph into two connected subgraphs.
    
    If $q_{kk'} = 0$, then can repeat our argument, removing cycles of the form $[k,k',k,k']$ or $[k',k,k',k]$ instead, and it gives a very similar proof as the one to \cite[Lemma 6.7]{warkentin_exchange_2014}.
\end{proof}

\section{Finite Pre-Forkless Part} \label{sec-fpfp}

\begin{Def} \label{def-pre-forkless-part}
    Let $Q$ be a quiver and $\Gamma$ be the labelled mutation graph of $Q$.
    Then we define the \textit{pre-forkless part} as the connected subgraph of $\Gamma$ satisfying the following: there are no quivers in the pre-forkless part that lie in a connected subgraph defined by Corollary \ref{cor-pre-fork-connected-subgraph} containing a pre-fork or by Lemma \ref{lem-fork-tree} containing a fork.
    If this is a finite set, we say that the quiver $Q$ has a \textit{finite pre-forkless part}.
    Furthermore, if the pre-forkless part has $m$ quivers in it, we say that $Q$ has a $m$ pre-forkless part.
\end{Def}

\begin{Rmk} \label{rmk-pre-forkless-part}
    Note that any quiver that has a finite forkless part naturally has a finite pre-forkless part as well.
    Additionally, we speak of \textit{the} pre-forkless part of the labelled mutation graph of $Q$, as there must be some $Q'$ in the pre-forkless part and $r \in Q_0$ such that $\mu_r(Q')$ is a pre-fork.
    Applying Corollary \ref{cor-pre-fork-connected-subgraph}, we see that the pre-forkless part is not only connected but convex and unique for every quiver.
\end{Rmk}

Lemma \ref{lem-key-subgraph} then gives us our first example of quivers with a finite pre-forkless part.

\begin{Cor} \label{cor-key-count}
    Let $K$ be a key with vertices $\{k,k'\}$, where $n \geq 4$ is the number of vertices in $K$.
    \begin{itemize}
        \item If $q_{kk'} = 0$, then $K$ is mutation-equivalent to a total of $n-1$ keys (including $K$), 2 acyclic quivers, $2n-6$ wings, and $n-3$ tips, totalling to $4n-8 = 4(n-2)$ quivers in its pre-forkless part.

        \item If $|q_{kk'}| = 1$, then $K$ is mutation-equivalent to a total of $2n-2$ keys (including $K$), 2 acyclic quivers, 4 non-acyclic quivers, $4n-12$ wings, and $4n-12$ tips, totalling to $10n-20 = 10(n-2)$ quivers in its pre-forkless part.
    \end{itemize}
    Thus $K$ must have a finite pre-forkless part.
\end{Cor}

We may then replicate the proof of Theorem \ref{thm-forkless-quiver}.

\begin{Thm} \label{thm-pre-forkless-quiver}
    Let $Q$ be a quiver that has a finite pre-forkless part.
    Then any full subquiver $Q'$ also has a finite pre-forkless part.
    Furthermore, if $M$ is the number of non-pre-forks mutation-equivalent to $Q$ and $m$ is the number of vertices in $Q'$, then the number of non-pre-forks mutation-equivalent to $Q'$ is bounded by $\max\{M,10m-20\}$.
\end{Thm}

\begin{proof}
    As $Q'$ is a full subquiver of $Q$, every quiver that is mutation-equivalent to $Q'$ is a full subquiver of a quiver mutation-equivalent to $Q$.
    If $Q'$ is not mutation-equivalent to a full subquiver of a pre-fork, then either $Q'$ has two vertices or it is only mutation-equivalent to subquivers of quivers in the pre-forkless part of $Q$.
    In either case, the quiver $Q'$ is mutation-finite and thus naturally has a finite pre-forkless part.
    If $Q'$ is mutation-equivalent to a full subquiver of a pre-fork, then it is either mutation-equivalent to an abundant acyclic quiver, a key, a fork, or a pre-fork by Lemma \ref{lem-pre-fork-subquivers}.
    
    In the first two cases, Corollaries \ref{cor-abundant-acyclic} and \ref{cor-key-count} tell us that $Q'$ must have finite pre-forkless part with at most $10m-20$ quivers.
    In the latter two cases, if $Q'$ is not mutation-equivalent to an abundant acyclic quiver or a key, then any quiver in its pre-forkless part must be a subquiver of a quiver in the pre-forkless part of $Q$.
    There are at most $M$, as there are only finitely many quivers in the pre-forkless part of $Q$.
    The subquiver $Q'$ must then have a finite pre-forkless part.
    Therefore, if $M$ is the number of non-pre-forks mutation-equivalent to $Q$ and $m$ is the number of vertices in $Q'$, then the number of non-pre-forks mutation-equivalent to $Q'$ is bounded by $\max\{M,10m-20\}$.
\end{proof}

Again, since having a finite pre-forkless part is a property of the labelled mutation class, we get the following corollary.

\begin{Cor} \label{cor-fpfp-hereditary}
    Having a finite pre-forkless part is a hereditary and mutation-invariant property. 
\end{Cor}

As both finite and mutation-finite quivers have a finite forkless part, this gives us the Venn diagram of hereditary and mutation-invariant properties in Figure \ref{fig-venn-diagram}.
This builds on Warkentin's remark on how we can throw away the forks and pre-forks we encounter in a mutation-class without losing any information. 
Indeed, he used this concept to show that the quivers
\[\begin{tikzcd}
i & j\\
\ell & k
\arrow[from=1-1, to=1-2, "a"]
\arrow[from=1-2, to=2-2, "b"]
\arrow[from=2-2, to=2-1, "a"]
\arrow[from=2-1, to=1-1, "b"]
\end{tikzcd}\]
for all $a,b \geq 2$ are mutation-cyclic \cite[Example 6.11]{warkentin_exchange_2014}.
They are the first non-trivial family of quivers that have finite pre-forkless part but not a finite forkless part.
This means that the quiver 
\[\begin{tikzcd}
i & j\\
\ell & k
\arrow[from=1-1, to=1-2, "2"]
\arrow[from=1-2, to=2-2, "2"]
\arrow[from=2-2, to=2-1, "2"]
\arrow[from=2-1, to=1-1, "2"]
\end{tikzcd}\]
is one of the simplest mutation-cyclic quivers on 4 vertices that does not allow a mutation-cyclic subquiver on 3 vertices to embed into it, which can be shown by observing the quivers in its pre-forkless part and the subquivers of tips, wings, pre-forks, and forks.
 
As in the forkless case, there is no lower bound on the number of quivers in the pre-forkless part.
For example, the quivers
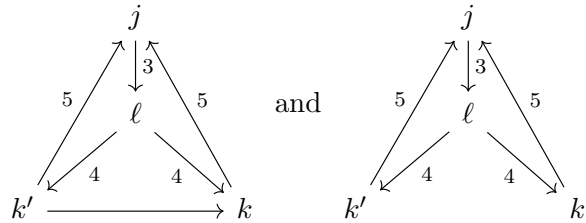
\[\begin{tikzcd}
& j\\
& \ell\\
k' & & k
\arrow[from=3-3, to=1-2, "5"']
\arrow[from=3-1, to=3-3]
\arrow[from=3-1, to=1-2, "5"]
\arrow[from=1-2, to=2-2, "3"]
\arrow[from=2-2, to=3-3, "4"']
\arrow[from=2-2, to=3-1, "4"]
\end{tikzcd} 
\text{ and }
\begin{tikzcd}
& j\\
& \ell\\
k' & & k
\arrow[from=3-3, to=1-2, "5"']
\arrow[from=3-1, to=1-2, "5"]
\arrow[from=1-2, to=2-2, "3"]
\arrow[from=2-2, to=3-3, "4"']
\arrow[from=2-2, to=3-1, "4"]
\end{tikzcd}\]
are both pre-forks with the same point of return $\ell$.
However, if we mutate at the point of return we get two more pre-forks with point of return $\ell$.
Naturally, this tells us that both our quivers have zero quivers in their pre-forkless part.
Each then has a 0 pre-forkless part.
However, neither of them has a finite forkless part, as alternating mutations at $j$ and $\ell$ will create an infinite number of pre-forks that are not abundant quivers.
We suspect that joining quivers with a 0 forkless part in a similar manner as above will always form quivers with a 0 pre-forkless part, but confirming that suspicion will require future research.
Since quivers with a finite forkless part have a finite pre-forkless part, our discussion from Section \ref{sec-forkless-part} demonstrates that there is no upper bound on the size of the pre-forkless part for quivers on 4 vertices. 

We also can carry over similar corollaries about disconnected quivers.
As disconnected quivers are not mutation-equivalent to forks by Corollary \ref{cor-connected-fork}, we naturally see that disconnected quivers are not mutation-equivalent to tips, wings, or pre-forks.

\begin{Cor} \label{cor-disconnected-pre-fork}
    A disconnected quiver is not mutation-equivalent to any tips, wings, or pre-forks.
\end{Cor}

Combined with Theorem \ref{thm-pre-forkless-quiver}, this naturally gives an analogue of Corollary \ref{cor-disconnected-finite}.

\begin{Cor}
    Let $Q$ be a quiver with a finite pre-forkless part.
    If $Q'$ is a disconnected full subquiver of $Q$, then $Q'$ must be mutation-finite.
\end{Cor}

Hence, the quiver 
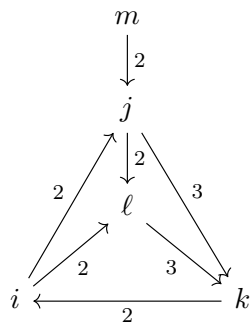
\[\begin{tikzcd}
& m \\
& j\\
& \ell\\
i & & k
\arrow[from=1-2, to=2-2, "2"]
\arrow[from=2-2, to=3-2, "2"]
\arrow[from=2-2, to=4-3, "3"]
\arrow[from=3-2, to=4-3, "3"']
\arrow[from=4-1, to=2-2, "2"]
\arrow[from=4-1, to=3-2, "2"']
\arrow[from=4-3, to=4-1, "2"]
\end{tikzcd}\]
from before cannot have a finite pre-forkless part for the exact same reasons it cannot have a finite forkless part.

\printbibliography

\end{document}